\documentclass[12 pt]{amsart}
\usepackage{amssymb,latexsym,amsmath,amscd,amsthm,graphicx, color}
\usepackage[all]{xy}
\usepackage{pgf,tikz}
\usepackage{mathrsfs}
\usepackage{cite}
\usetikzlibrary{arrows}
\definecolor{uuuuuu}{rgb}{0.26666666666666666,0.26666666666666666,0.26666666666666666}
\definecolor{xdxdff}{rgb}{0.49019607843137253,0.49019607843137253,1.}
\definecolor{ffqqqq}{rgb}{1.,0.,0.}

\definecolor{uuuuuu}{rgb}{0.26666666666666666,0.26666666666666666,0.26666666666666666}
\definecolor{qqwuqq}{rgb}{0.,0.39215686274509803,0.}
\definecolor{zzttqq}{rgb}{0.6,0.2,0.}
\definecolor{xdxdff}{rgb}{0.49019607843137253,0.49019607843137253,1.}
\definecolor{qqqqff}{rgb}{0.,0.,1.}
\definecolor{cqcqcq}{rgb}{0.7529411764705882,0.7529411764705882,0.7529411764705882}

\theoremstyle{plain}

\newtheorem{theorem}[subsection]{Theorem}

\newtheorem{lemma}[subsection]{Lemma}

\theoremstyle{definition}
\newtheorem{prop}[subsection]{Proposition}
\newtheorem{cor}[subsection]{Corollary}

\newtheorem{remark}[subsection]{Remark}

\newtheorem{notation}[subsection]{Notation}
\newtheorem{note}[subsection]{Note}

\newcommand{\uu}{\cup}
\newcommand{\ii}{\cap}
\newcommand{\UU}{\bigcup}
\newcommand{\II}{\bigcap}

\newcommand{\sci}{\subset}
\newcommand{\es}{\emptyset}
\newcommand{\set}[1]{\{#1\}}

\newcommand{\ga}{\alpha}
\newcommand{\gb}{\beta}

\newcommand{\gd}{\delta}
\renewcommand{\gg}{\gamma}

\newcommand{\go}{\omega}

\newcommand{\gs}{\sigma}
\newcommand{\gt}{\tau}

\newcommand{\tit}{\textit}

\newcommand{\C}[1]{\mathcal{#1}}
\newcommand{\D}[1]{\mathbb{#1}}

\newcommand{\te}{\text}

\newcommand{\nd}{\noindent}

\begin{document}
\nd To appear, Communications of the Korean Mathematical Society
\title[Quantization for a probability distribution generated by an infinite IFS]
{Quantization for a probability distribution generated by an infinite iterated function system}

\author[L. Roychowdhury]{Lakshmi Roychowdhury}
\address{Lakshmi Roychowdhury \\ School of Mathematical and Statistical Sciences\\
University of Texas Rio Grande Valley\\
 1201 West University Drive\\
Edinburg, TX 78539, USA}
\email{lakshmiroychowdhury82@gmail.com}

\author[M.K. Roychowdhury]{Mrinal Kanti Roychowdhury}
\address{Mrinal Kanti Roychowdhury \\ School of Mathematical and Statistical Sciences\\
University of Texas Rio Grande Valley\\
 1201 West University Drive\\
Edinburg, TX 78539, USA}
\email{mrinal.roychowdhury@utrgv.edu}
\thanks{The research of the second author was supported by U.S. National Security Agency (NSA) Grant H98230-14-1-0320}

\pagestyle{myheadings}\markboth{Lakshmi Roychowdhury and  Mrinal Kanti Roychowdhury}{Quantization for a probability distribution generated by an infinite IFS}
\begin{abstract}
Quantization for probability distributions concerns the best approximation of a $d$-dimensional probability distribution $P$ by a discrete probability with a given number $n$ of supporting points. In this paper, we have considered a probability measure generated by an infinite iterated function system associated with a probability vector on $\mathbb R$. For such a probability measure $P$, an induction formula to determine the optimal sets of $n$-means and the $n$th quantization error for every natural number $n$ is given. In addition, using the induction formula we give some results and observations about the optimal sets of $n$-means for all $n\geq 2$.

\end{abstract}

\maketitle

\section{Introduction}
Quantization is the process of converting a continuous analog signal into
a digital signal of $k$ discrete levels, or converting a digital signal of $n$
levels into another digital signal of $k$ levels, where $k < n$. It is must
when analog quantities are represented, processed, stored, or transmitted
by a digital system, or when data compression is required. It is a classic
and still very active research topic in source coding and information
theory. A good survey about the historical development of the theory has been provided by Gray and Neuhoff in \cite{GN}. For more applied aspects of quantization the reader is referred to the book of Gersho and Gray (see \cite{GG}). For mathematical treatment of quantization one may consult Graf-Luschgy's book (see \cite{GL2}). Interested readers can also see \cite{AW, GKL, GL1, Z}. Let $\D R^d$ denote the $d$-dimensional Euclidean space equipped with the Euclidean metric $\|\cdot\|$. Let $P$ be a Borel probability measure on $\D R^d$. Then, the $n$th quantization error for $P$, denoted by $V_n:=V_n(P)$, is defined by
\[V_n(P)=\inf_{\ga \in \C D_n} \int \min_{a \in \ga} \|x-a\|^2 dP(x),\]
where $\C D_n:=\set{\ga \sci \D R^d : 1\leq \te{card}(\ga)\leq n}$. The set $\ga$ for which the infimum occurs and contains no more than $n$ points is called an optimal set of $n$-means for $P$, and such a set exists if $\int\|x\|^2 dP(x)<\infty$ (see \cite{ GKL, GL1, GL2}). The set of all optimal sets of $n$-means for a probability measure $P$ is denoted by $\C C_n(P)$. It is known that for a Borel probability measure $P$ if the support of $P$ contains infinitely many elements, then an optimal set of $n$-means always has exactly $n$-elements (see \cite[Theorem 4.12]{GL2}). Let $\ga$ be a finite set and $a\in \ga$. Then, the \tit{Voronoi cell}, or \tit{Voronoi region} $M(a|\ga)$ is the set of all elements in $\D R^d$ whose distance to $a$ is not greater than their distance to other elements in $\ga$, i.e.,
\[M(a|\ga)=\set{x \in \D R^d : \|x-a\|=\min_{b \in \ga}\|x-b\|}.\] A Borel measurable partition $\set{A_a : a \in \ga}$ of $\D R^d$  is called a \tit{Voronoi partition} of $\D R^d$ with respect to $\ga$ (and $P$) if
$A_a \sci M(a|\ga) \te{ ($P$-a.e.) for every $a \in \ga$}.$
 The following proposition is known (see \cite{GG, GL1}).
\begin{prop} \label{prop0}
Let $\alpha$ be an optimal set of $n$-means, $a \in \ga$, and $M (a|\ga)$ be the Voronoi region generated by $a\in \ga$. Then, for every $a \in\ga$,

$(i)$ $P(M(a|\ga))>0$, $(ii)$ $ P(\partial M(a|\ga))=0$, $(iii)$ $a=E(X : X \in M(a|\ga))$, and $(iv)$ $P$-almost surely the set $\set{M(a|\ga) : a \in \ga}$ forms a Voronoi partition of $\D R^d$.
\end{prop}
Since for $a\in \ga$,  $a=E(X : X\in M(a|\ga))=\frac{1}{P(M(a|\ga))}\int_{M(a|\ga)} xdP(x)$, we can say that the elements in an optimal set of $n$-means are also the centroids of their own Voronoi regions with respect to the probability distribution $P$. For details in this regard one can see \cite{DFG, R3}.

Let $M$ denote either the set $\set{1, 2, \cdots, N}$ for some positive integer $N\geq 2$, or the set $\D N$ of natural numbers. A collection $\set{S_j : j\in M}$ of similarity mappings, or similitudes, on $\D R^d$ with similarity ratios $\set{s_j : j\in M}$ is contractive if $\sup\set{s_j : j\in M}<1$. If $J$ is the limit set of the iterated function system then it is known that $J$ satisfies the following invariance relation (see \cite{H, MaU, M}):
\[J=\UU_{j\in M} S_j(J).\]  The iterated function system $\set{S_j: j \in M}$ satisfies the \tit{open set condition} (OSC),  if there exists a bounded nonempty open set $U \sci \D R^d$ such that $S_j(U) \sci U$ for all $j\in M$, and $S_i(U) \II S_j(U) =\es$ for $i, j\in M$ with $i\neq j$.
Let $(p_j: j\in M)$ be a probability vector, with $p_j>0$ for all $j\in M$. Then, there exists a unique Borel probability measure $P$ on $\D R^d$
(see \cite{H}, \cite{MaU}, \cite{M}, etc.), such that
\begin{equation*} \label{eq10} P=\sum_{j\in M} p_j P \circ S_j^{-1},\end{equation*}
where $P\circ S_j^{-1}$ denotes the image measure of $P$ with respect to
$S_j$ for $j\in M$.
Such a $P$ has support the limit set $J$ if $M$ is finite, or the closure of $J$ if $M$ is infinite.

Let $P$ be a Borel probability measure on $\D R$ generated by the two contractive similarity mappings $S_1$ and $S_2$ associated with the probability vector $(\frac 12, \frac 12)$ such that $S_1(x)=\frac 13x$ and $S_2(x)=\frac 13 x +\frac 23$ for all $x\in \D R$.  Then, $P=\frac 12 P\circ S_1^{-1}+\frac 12 P\circ S_2^{-1}$ and it has support the classical Cantor set generated by $S_1$ and $S_2$. For this probability measure Graf and Luschgy gave a closed formula to determine the optimal sets of $n$-means and the $n$th quantization errors for all $n\geq 2$ (see \cite{GL3}).
Later for $n\geq 2$, L. Roychowdhury gave an induction formula to determine the optimal sets of $n$-means and the $n$th quantization errors for a probability distribution $P$ on $\D R$, given by $P=\frac 14 P\circ S_1^{-1}+\frac 34 P\circ S_2^{-1}$ which has support the Cantor set generated by $S_1$ and $S_2$, where $S_1(x)=\frac 14 x$ and $S_2(x)=\frac 12 x+\frac 12$ for all $x \in \D R$ (see \cite{R1}). M. Roychowdhury (see \cite{R2}) gave an infinite extension of the result of Graf-Luschgy (see \cite{GL3}). \c C\"omez and Roychowdhury (see \cite{CR}) gave a closed formula to determine the optimal sets of $n$-means and the $n$th quantization error for a probability measure supported by a Cantor dust.

In this paper, we made an infinite extension of the work of L. Roychowdhury (see \cite{R1}). Let $P$ be a Borel probability measure on $\D R$ given by $P=\frac 1 4 P\circ S_1^{-1}+\sum_{j=2}^\infty \frac 3 {2^{j+1}} P\circ S_j^{-1}$, i.e., $P$ is generated by an infinite collection of similitudes $\set{S_j}_{j=1}^\infty$ associated with the probability vector $(\frac 14, \frac {3}{2^3}, \frac {3}{2^4}, \cdots)$ such that $S_j(x)=\frac 1 {2^{j+1}} x+1-\frac{1}{2^{j-1}}$ for all $x\in\D R$, and for all $j\in \D N$.  For this probability measure, in this paper, we investigate the optimal sets of $n$-means and the $n$th quantization errors for all $n\in \D N$.
The arrangement of the paper is as follows: In Lemma~\ref{lemma11} and Lemma~\ref{lemma12}, we obtain the optimal sets of $n$-means and the corresponding quantization errors for $n=2$ and $n=3$; Proposition~\ref{prop1}, Proposition~\ref{prop2}, Proposition~\ref{prop3}, and Proposition~\ref{prop50} give some properties about the optimal sets of $n$-means and the $n$th quantization errors. In Theorem~\ref{Th1} we state and prove an induction formula to determine the optimal sets of $n$-means for all $n\geq 2$. In addition, using the induction formula we obtain some results and observations about the optimal sets of $n$-means which are given in Section~4; a tree diagram of the optimal sets of $n$-means for a certain range of $n$ is also given.

\section{Preliminaries}

By a word $\go$ over the set $\D N=\set{1, 2, 3, \cdots}$ of natural numbers it is meant that $\go:=\go_1\go_2\cdots \go_k\in \D N^k$ for some $k\geq 1$. Here $k$ is called the length of the word $\go$ and is denoted by $|\go|$. A word of length zero is called the empty word and is denoted by $\es$. Let $\D N^\ast$ denote the set of all words over the alphabet $\D N$ including the empty word $\es$. For any two words $\go:=\go_1\go_2\cdots \go_k$ and $\gt:=\gt_1\gt_2\cdots\gt_m \in \D N^\ast$, where $k, m\geq 1$, by $\go\gt$ it is meant the concatenation of the words $\go$ and $\gt$, i.e., $\go\gt=\go_1\go_2\cdots \go_k\gt_1\gt_2\cdots\gt_m$. If $\go:=\go_1\go_2\cdots \go_k$, we write $\go^-:=\go_1\go_2\cdots\go_{k-1}$ where $k\geq 1$, i.e., $\go^-$ is the word obtained from the word $\go$ by deleting the last letter of $\go$. For $\go \in \D N^\ast$, by $(\go, \infty)$ it is meant the set of all words $\go^-(\go_{|\go|}+j)$, obtained by concatenation of the word $\go^-$ with the word $\go_{|\go|}+j$ for $j\in \D N$, i.e.,
\[(\go, \infty)=\set{\go^-(\go_{|\go|}+j) : j\in \D N}.\]
Let $(p_j)_{j=1}^\infty$ be a probability vector such that $p_1=\frac 14$ and $p_j=\frac 3 {2^{j+1}}$ for all $j\geq 2$. Let $\set{S_j}_{j=1}^\infty$ be an infinite collection of similitudes associated with the probability vector $(p_j)_{j=1}^\infty$ such that
\[S_j(x)=\frac 1{2^{j+1}} x+1-\frac 1{2^{j-1}} \]
for all $j \in \D N$ and for all $x\in \D R$. Then, as mentioned in the previous section, there exists a unique Borel probability measure $P$ on $\D R$ such that
\[P=\sum_{j=1}^\infty  p_j P\circ S_j^{-1},\]
which has support lying in the closed interval $[0, 1]$. This paper deals with this probability measure $P$.
For $\go=\go_1\go_2\cdots\go_n\in \D N^n$, write
\[S_\go:=S_{\go_1}\circ\cdots \circ S_{\go_n}, \quad J_\go:=S_\go(J), \quad s_\go:=s_{\go_1}\cdots s_{\go_n}, \quad p_\go:=p_{\go_1}\cdots p_{\go_n},\]
where $J:=J_\es=[0, 1]$. We also assume $p_\es=1$ and $s_\es=1$. Then, for any $\go \in \D N^\ast$, we write
\[J_{(\go,\infty)}:=\mathop{\uu}\limits_{j=1}^\infty J_{{\go^-(\go_{|\go|}+j)}} \te{ and }      p_{(\go, \infty)}:=P(J_{(\go, \infty)})=\sum_{j=1}^\infty P(J_{\go^-(\go_{|\go|}+j)})=\sum_{j=1}^\infty p_{\go^-(\go_{|\go|}+j)}.\]
Notice that for any $k\in \D N$, $p_{(k, \infty)}=1-\sum_{j=1}^kp_j$, and for any word $\go \in \D N^\ast$, $p_{(\go, \infty)}=p_{\go^-}-\sum_{j=1}^{w_{|\go|}}p_{\go^-j}$. To avoid any confusion among the readers, we would like to mention that in the paper $dP(x)$ which is $P(dx)$ is identified as $dP$.

\begin{lemma} \label{lemma1}
Let $f : \mathbb R \to \mathbb R^+$ be Borel measurable and $k\in \mathbb N$. Then
\[\int f dP=\sum_{\go \in \D N^k}p_\go \int f \circ S_\go dP.\]
\end{lemma}

\begin{proof}
We know $P=\sum_{j=1}^\infty p_j P\circ S_j^{-1}$, and so by induction $P= \sum_{\go \in \D N^k} p_\go P\circ S_\go^{-1}$, and thus the lemma is yielded.
\end{proof}

\begin{lemma} \label{lemma2} Let $X$ be a random variable with probability distribution $P$. Then, the expectation $E(X)$ and the variance $V:=V(X)$ of the random variable $X$ are given by
\[E(X)=\frac {4}{7}  \text{ and } V(X)=\frac{288}{3577}=0.0805144.\]
\end{lemma}
\begin{proof} Using Lemma~\ref{lemma1}, we have
\begin{align*}
&E(X)=\int x dP=\frac 1 4 \int S_1(x) dP+\sum_{j=2}^\infty \frac 3{2^{j+1}} \int S_j(x) dP\\
&=\frac{1}{16} \int x dP+\sum_{j=2}^\infty \frac 3{2^{j+1}} \int\Big(\frac 1{2^{j+1}} x+1-\frac 1{2^{j-1}}\Big) dP=\frac1{16} E(X)+\frac1{16} E(X)+ \frac 12,
\end{align*}
which implies $E(X)=\frac 47$. Now,
\begin{align*}
&E(X^2)=\int x^2 dP= \frac 14 \int (\frac 14 x)^2 dP+\sum_{j=2}^\infty \frac 3{2^{j+1}} \int\Big(\frac 1{2^{j+1}} x+1-\frac 1{2^{j-1}}\Big)^2 dP\\
&=\frac 1{64}E(X^2)+\sum_{j=2}^\infty\frac 3{2^{j+1}}  \int \Big(\frac 1{4^{j+1}} x^2 + \frac 2{2^{j+1}} (1-\frac{1}{2^{j-1}}) x+(1-\frac 1{2^{j-1}})^2 \Big)dP\\
&=\frac 1{64}E(X^2)+\frac{3}{448} E(X^2)+\frac 1{14} E(X)+\frac 5 {14}=\frac{5}{224} E(X^2)+\frac{39}{98},
\end{align*}
which yields $E(X^2)=\frac{208}{511}$. Thus,
$V(X)=E(X^2)-\left(E(X)\right)^2 =\frac{288}{3577}=0.0805144,$
which is the lemma.
\end{proof}

\begin{lemma} \label{lemma3}
For any $k\geq 2$, we have
\[E(X | X \in J_k\uu J_{k+1}\uu \cdots)=1-\frac{8}{7}\frac 1{2^k}.\]
\end{lemma}
\begin{proof}We have
\begin{align*}
&E(X | X \in J_k\uu J_{k+1}\uu \cdots)=\frac {1}{\sum_{j=k}^\infty p_j} \sum_{j=k}^\infty p_j S_j(\frac {4}{7})=
 \frac {2^k}{3} \Big(\sum_{j=k}^\infty \frac 3 {2^{j+1}}(\frac 1{2^{j+1}} \frac 4 7+1-\frac 1{2^{j-1}})\Big),
\end{align*}
which after simplification yields
$E(X | X \in J_k\uu J_{k+1}\uu \cdots)=1-\frac{8}{7}\frac 1{2^k},$
which is the lemma.
\end{proof}

The following notes are in order.

\begin{note} \label{note1} For $k\in \D N$, we have
$S_k(\frac 47)=\frac 1{2^{k+1}}\frac 47 +1-\frac 1{2^{k-1}}.$ Thus, by Lemma~\ref{lemma3}, for $k\in \D N$,
\begin{equation*}  E(X | X \in J_k\uu J_{k+1}\uu \cdots)=S_k(\frac 47) +\frac 1 7\frac 1{2^{k-2}}=S_k(\frac 47)+\frac 8 7 s_k.
\end{equation*}  Since for any $x_0 \in \D R$,
$\int(x-x_0)^2 dP =V(X)+(x_0-E(X))^2$, we can deduce that the optimal set of one-mean is the expected value and the corresponding quantization error is the variance $V$ of the random variable $X$. For $\go \in \D N^k$, $k\geq 1$, using Lemma~\ref{lemma1}, we have
\begin{align*}
&E(X : X \in J_\go) =\frac{1}{P(J_\go)} \int_{J_\go} xdP=\int_{J_\go} x  d(P\circ S_\go^{-1}(x))=\int S_\go(x)  dP=E(S_\go(X)).
\end{align*}
Since $S_j$ are similitudes, it is easy to see that $E(S_j(X))=S_j(E(X))$ for $j\in \D N$, and so by induction, $E(S_\go(X))=S_\go(E(X))$
 for $\go\in \D N^k$, $k\geq 1$.
\end{note}

\begin{note}

For words $\gb, \gg, \cdots, \gd$ in $\D N^\ast$, by $a(\gb, \gg, \cdots, \gd)$ we denote the conditional expectation of the random variable $X$ given that $X$ is in $J_\gb\uu J_\gg \uu\cdots \uu J_\gd,$ i.e.,
\begin{equation} \label{eq90} a(\gb, \gg, \cdots, \gd)=E(X|X\in J_\gb \uu J_\gg \uu \cdots \uu J_\gd)=\frac{1}{P(J_\gb\uu \cdots \uu J_\gd)}\int_{J_\gb\uu \cdots \uu J_\gd} x dP.
\end{equation}
Then, by Note~\ref{note1}, for $\go\in \D N^\ast$, we have
\begin{align} \label{eq1} \left\{ \begin{array}{ll} a(\go)=S_\go(E(X))=S_\go(\frac 47), \te{ and  } & \\
 a(\go, \infty)=E(X|X\in J_{\go^-(\go_{|\go|}+1)}\uu J_{\go^-(\go_{|\go|}+2)}\uu \cdots)=S_{\go^-(\go_{|\go|}+1)}(\frac 47)+\frac 87 s_{\go^-(\go_{|\go|}+1)}.&
 \end{array} \right.
 \end{align}
Moreover, for any  $\go \in \D N^\ast$  and for any $x_0\in \D R$, it is easy to see that
\begin{align} \label{eq2} \left\{ \begin{array}{ll}\int_{J_\go}(x-x_0)^2 dP=p_\go\int (x -x_0)^2 d(P\circ S_\go^{-1})=p_\go  \Big(s_\go^2V+(S_\go(\frac 47)-x_0)^2\Big),\te{ and } &\\
\int_{J_{(\go, \infty)}} (x -x_0)^2  dP=\sum_{j=1}^\infty p_{\go^-(\go_{|\go|}+j)} \Big(s_{\go^-(\go_{|\go|}+j)}^2 V + (S_{\go^-(\go_{|\go|}+j)}(\frac 47)-x_0)^2\Big).
\end{array} \right.
 \end{align}
The expressions \eqref{eq1} and \eqref{eq2}  are useful to obtain the optimal sets and the corresponding quantization errors with respect to the probability distribution $P$.
\end{note}

The following lemma plays a vital role in the paper.
\begin{lemma} \label{lemma4} Let $P$ be the probability measure as defined before and let $\go\in \D N^k$, $k\geq 1$. Then,
\[\int_{J_{\go}} (x-a(\go))^2  dP=p_\go s_\go^2 V, \te{ and } \int_{J_{(\go, \infty)}} (x-a(\go, \infty))^2  dP=\left\{\begin{array}{cc} \frac{43}{9} p_\go s_\go^2V\te{ if } \go_{|\go|}\geq 2, &\\
\frac{43}{3} p_\go s_\go^2V \te{ if } \go_{|\go|}=1. &
\end{array}\right.\]
\end{lemma}

\begin{proof} In the first equation of \eqref{eq2} put $x_0=a(\go)$, and then $\int_{J_{\go}} (x-a(\go))^2  dP=p_\go s_\go^2 V$. In the second equation of \eqref{eq2}, put $x_0=a(\go, \infty)$, and then
\begin{align} \label{eq36} & \int_{J_{(\go, \infty)}} (x -a(\go, \infty))^2  dP\\
&=\sum_{j=1}^\infty p_{\go^-(\go_{|\go|}+j)}s_{\go^-(\go_{|\go|}+j)}^2 V + \sum_{j=1}^\infty p_{\go^-(\go_{|\go|}+j)}\Big(S_{\go^-(\go_{|\go|}+j)}(\frac 47)-a(\go, \infty)\Big)^2.\notag
\end{align}
Putting the values of $a(\go, \infty)$ from \eqref{eq1} we have
\begin{align*}
&S_{\go^-(\go_{|\go|}+j)}(\frac 47)-a(\go, \infty)=S_{\go^-(\go_{|\go|}+j)}(\frac 47)-S_{\go^-(\go_{|\go|}+1)}(\frac 47)-\frac 87 s_{\go^-(\go_{|\go|}+1)}\\
&=s_{\go^-} \Big(S_{\go_{|\go|}+j}(\frac 47)-S_{\go_{|\go|}+1}(\frac 47)-\frac 87 s_{\go_{|\go|}+1}\Big)\\
&=s_{\go^-}\Big(\frac{1}{2^{\go_{|\go|+j+1}}} \frac 47-\frac{1}{2^{\go_{|\go|+j-1}}}-\frac{1}{2^{\go_{|\go|+1+1}}} \frac 47+\frac{1}{2^{\go_{|\go|+1-1}}}-\frac 87 s_{\go_{|\go|}+1}\Big)\\
&=s_{\go}\Big(\frac{1}{2^{j}} \frac 47-\frac{4}{2^{j}}- \frac 27+2-\frac 47\Big)=s_\go\Big(\frac{8}{7}-\frac{24}{7} \frac 1{ 2^j}\Big).
\end{align*}
Moreover, for any $j\geq 1$, $s_{\go^-(\go_{|\go|}+j)}=s_\go \frac 1{2^j}$; and $p_{\go^-(\go_{|\go|}+j)}= p_\go \frac 1{2^j} \te{ if } \go_{|\go|}\geq 2$, and $ p_{\go^-(\go_{|\go|}+j)}=p_\go \frac 3{2^j} \te{ if } \go_{|\go|}=1.$ Thus if $\go_{|\go|}\geq 2$, putting the corresponding values and making some simplification, we obtain
\begin{align*} &\sum_{j=1}^\infty p_{\go^-(\go_{|\go|}+j)}s_{\go^-(\go_{|\go|}+j)}^2 V=\frac 1 {7} p_\go s_\go^2V \te{ and } \\
 &\sum_{j=1}^\infty p_{\go^-(\go_{|\go|}+j)}\Big(S_{\go^-(\go_{|\go|}+j)}(\frac 47)-a(\go, \infty)\Big)^2=p_\go s_\go^2 \sum_{j=1}^\infty \frac 1 {2^j}\Big(\frac{8}{7}-\frac{24}{7} \frac 1{ 2^j}\Big)^2=p_\go s_\go^2V \frac{292}{63},
\end{align*}
and then \eqref{eq36} yields
$\int_{J_{(\go, \infty)}} (x -a(\go, \infty))^2  dP=\frac {43} 9 p_\go s_\go^2 V$. Similarly, if $\go_{|\go|}=1$, one can obtain $\int_{J_{(\go, \infty)}} (x -a(\go, \infty))^2  dP=\frac {43} 3 p_\go s_\go^2 V$. Thus, the lemma is yielded.
\end{proof}

\begin{notation}
For any $\go \in \D N^k$, $k\geq 1$, set
\begin{align}\label{eq43}
E(a(\go)):=\int_{J_{\go}}(x-a(\go))^2 dP \te{ and } E(a(\go, \infty)):=\int_{J_{(\go, \infty)}}(x-a(\go, \infty))^2 dP.
\end{align}
\end{notation}
Let us now prove the following lemma.

\begin{lemma} \label{lemma5}
For any two nonempty words $\go, \gt\in \D N^\ast$ if $p_\go=p_\gt$, then $s_\go=s_\gt$.
\end{lemma}

\begin{proof} To prove the lemma, let us define a function $c$ as follows:
\[c: \D N^\ast\setminus \set{\es} \to \D N\uu\set{0}, \te{ such that } c(\go)=\te{card}(\set{\go_i : \go_i\neq 1, \, 1\leq i\leq |\go|}).\]
Let $\go, \gt \in \D N^\ast$ with $\go=\go_1\go_2\cdots \go_k$ and $\gt=\gt_1\gt_2\cdots \gt_m$ for some $k, m\geq 1$. Then, $p_\go=p_\gt$ implies
\[\frac{3^{c(\go)}}{2^{\go_1+\go_2+\cdots +\go_k+k}}=\frac{3^{c(\gt)}}{2^{\gt_1+\gt_2+\cdots +\gt_m+m}}\]
yielding $3^{c(\go)-c(\gt)}=2^{(\go_1+\go_2+\cdots +\go_k+k)-(\gt_1+\gt_2+\cdots +\gt_m+m)}$ and so, $c(\go)=c(\gt)$ and $\go_1+\go_2+\cdots +\go_k+k=\gt_1+\gt_2+\cdots +\gt_m+m$. Then,
\[s_\go=\frac{1}{2^{\go_1+\go_2+\cdots +\go_k+k}}=\frac{1}{2^{\gt_1+\gt_2+\cdots +\gt_m+m}}=s_\gt,\]
which is the lemma.
\end{proof}

In the next section we state and prove the main result of the paper.

\section{Main Result}
The following theorem gives the main result of the paper.

\begin{theorem} \label{Th1} For any $n\geq 2$, let $\ga_n:=\set{a(i) : 1\leq  i\leq n}$ be an optimal set of $n$-means, i.e., $\ga_n \in\C C_n:= \mathcal{C}_n(P)$. For $\go \in \D N^k$, $k\geq 1$, let $E(a(\go))$ and $E(a(\go, \infty))$ be defined by \eqref{eq43}. Set
\[\tilde  E(a(i)):=\left\{\begin{array} {ll}
E(a(\go)) \te{ if } a(i)=a(\go) \te{ for some }  \go \in \D N^\ast, \\
E(a(\go, \infty)) \te{ if } a(i)=a(\go, \infty) \te{ for some }  \go \in \D N^\ast,
\end{array} \right.
\]
and $W(\ga_n):=\set{a(j)  : a(j) \in \ga_n \te{ and } \tilde E(a(j))\geq \tilde E(a(i)) \te{ for all } 1\leq i\leq n}$. Take any $a(j) \in W(\ga_n)$, and write
\[\ga_{n+1}(a(j)):=\left\{\begin{array}{ll}
(\ga_n\setminus \set{a(j)})\uu \set{a(\go^-(\go_{|\go|}+1)), \, a(\go^-(\go_{|\go|}+1), \infty)} \te{ if } a(j)=a(\go, \infty), &\\
(\ga_n \setminus \set{a(j)})\uu \set{a(\go1), \, a(\go1, \infty)} \te{ if } a(j)=a(\go), &
\end{array}\right.
\]
Then, $\ga_{n+1}(a(j))$ is an optimal set of $(n+1)$-means, and the number
of such sets is given by
\[\te{card}\Big(\UU_{\ga_n \in \C{C}_n}\{\ga_{n+1}(a(j)) : a(j) \in W(\ga_n)\}\Big).\]
\end{theorem}

\begin{remark}
Once an optimal set of $n$-means is known, by using \eqref{eq2} and Lemma~\ref{lemma4}, the corresponding quantization error can easily be calculated.
\end{remark}
To prove Theorem~\ref{Th1} we need some basic lemmas and propositions.

\begin{lemma}\label{lemma11} Let $\ga=\{a_1, a_2\}$ be an optimal set of two-means, $a_1<a_2$. Then, $a_1=a(1)=\frac 17$, $a_2=a(1, \infty)=\frac 57$ and the quantization error is $V_2=\frac{69}{3577}=0.0192899$.

\end{lemma}

\begin{proof} Let us first consider the two-point set $\beta$ given by $\gb=\{\frac 17, \frac 57\}$. Since $S_1(1) <\frac 12 (\frac 1 7+\frac 57)<S_2(0)$, by Lemma~\ref{lemma4}, we have
\begin{align*}
&\int \min_{b \in \gb}(x-b)^2  dP=\int_{J_1} (x-\frac 17)^2 dP+\int_{J_{(1, \infty)}} (x-\frac 57)^2 dP\\
&=p_1s_1^2  (1+\frac{43}{3}) V=\frac{69}{3577}=0.0192899.
\end{align*}
Since $V_2$ is the quantization error for two-means, we have $V_2\leq 0.0192899$.
Let $\ga=\{a_1, a_2\}$ be an optimal set of two-means, $a_1<a_2$. Since $a_1$ and $a_2$ are the centroids of their own Voronoi regions, we have $0< a_1<a_2< 1$. Suppose that $a_2\leq \frac 58$. Then,
\[V_2\geq \int_{J_3\uu J_4\uu J_5\uu J_6}(x-\frac 58)^2 dP=\frac{647055}{33488896}=0.0193215>V_2,\]
which leads to a contradiction. So, we can assume that $a_2>\frac 58$ implying $\frac 12(a_1+a_2)\geq \frac 12(0+\frac 58)=\frac 5{16}>\frac 14$. Thus, we see that the Voronoi region of $a_2$ does not contain any point from $J_1$, and  $a_1\geq a(1)=\frac 17$.
Suppose that $a_1\geq \frac 7{16}$.
Then, using \eqref{eq2}, we have
\begin{align*}
V_2 \geq \int_{J_1}(x-a_1)^2 dP\geq \int_{J_1}(x-\frac 7{16})^2 dP = p_1 \Big(s_1^2 V+(S_1(\frac 47)-\frac 7{16})^2\Big)=\frac{12015}{523264}=0.0229616>V_2
\end{align*}
which is a contradiction, and so $\frac 17\leq a_1<\frac 7{16}$.
We now show that $\frac 12(a_1+a_2)\leq \frac 12$. For the sake of contradiction assume that $\frac 12(a_1+a_2)>\frac 12$. Then, if $\frac 12(a_1+a_2)\geq \frac 58$, we have
$a_1\geq E(X : X \in J_1\uu J_2)=\frac 25,$
yielding \[V_2\geq \int_{J_1\uu J_2}(x-\frac 25)^2 dP=\frac{171}{5840}=0.0292808>V_2,\] which is a contradiction. Next, assume that $S_{2\gs1}(1)\leq \frac 12(a_1+a_2)\leq S_{2\gs2}(0)$ for some $\gs \in \D N^\ast$. For definiteness sake, take $\gs=1$, and so $S_{211}(1)\leq \frac 12(a_1+a_2)\leq S_{212}(0)$. Then, $a_1=E(X : X\in J_1\uu J_{211})$ and $a_2=E(X : X\in J_{(211, \infty)}\uu J_{(21, \infty)}\uu J_{(2, \infty)})$ yielding
\begin{align*}
a_1&=\frac{P(J_1)S_{1}(\frac 47)+P(J_{211})S_{211}(\frac 47)}{P(J_1)+P(J_{211})}=\frac{1363}{7840},
\end{align*} \te{ and }
\begin{align*}
a_2&=\frac{p_{(211, \infty)}a(211, \infty)+p_{(21, \infty)}a(21, \infty)+p_{(2, \infty)} a(2, \infty)}{p_{(211, \infty)}+p_{(21, \infty)}+p_{(2, \infty)}}=\frac{5007}{6944},
\end{align*}
where $p_{(211, \infty)}=p_{21}-p_{211}$, $p_{(21, \infty)}=p_2-p_{21}$, $p_{(2, \infty)}=1-p_1-p_2$,  $a(211, \infty)=S_{212}(\frac 47)+\frac 87 s_{212}$, $a(21, \infty)=S_{22}(\frac 47)+\frac 87 s_{22}$, and $a(2, \infty)=S_3(\frac 47)+\frac 87 s_3$.
Thus,
\[V_2\geq \int_{J_1\uu J_{211}}(x-\frac{1363}{7840})^2dP +\int_{A}(x-\frac{5007}{6944})^2 dP=\frac{648995235322779}{32296112614277120}=0.0200952>V_2,\]
where $A=J_{212}\uu J_{213}\uu J_{22}\uu J_{23}\uu J_{24}\uu J_{25}\uu J_3\uu J_4\uu J_5\uu J_6\uu J_7\uu J_8\uu J_9\uu J_{10}$, which gives a contradiction. Similarly, we can show that for any other choice of $\gs\in \D N^\ast$, the assumption $\frac 12(a_1+a_2)>\frac 12$ will give a contradiction. Thus, we have $\frac 12(a_1+a_2)\leq \frac 12$ implying $a_1\leq a(1)=\frac 17$. Again, we have seen $a_1\geq \frac 17$. Thus, we deduce that $a_1=\frac 17$ and the Voronoi region of $a_2$ does not contain any point from $J_1$, i.e.,  $a_2=a(1, \infty)=\frac 57$, and the corresponding quantization error is $V_2=\frac{69}{3577}=0.0192899$. This completes the proof of the lemma.
\end{proof}

Using the technique of Lemma~\ref{lemma11}, the following corollary can be proved.

\begin{cor} \label{cor2}
For any $\go\in \D N^\ast$,  the set $\set{a(\go1), a(\go1, \infty)}$ forms a unique optimal set two-means for the conditional measure of $P$ on $J_\go$, and the set $\set{a(\go^-(\go_{|\go|}+1)), a(\go^-(\go_{|\go|}+1), \infty)}$ forms a unique optimal set of two-means for the conditional measure of $P$ on $J_{(\go, \infty)}$.
\end{cor}

\begin{figure}
\begin{tikzpicture}[line cap=round,line join=round,>=triangle 45,x=1.0cm,y=1.0cm]
\clip(-0.3,3.3) rectangle (16.3,4.7);
\draw [dotted] (0.,4.)-- (16.,4.);
\draw (0.,4.)-- (4.,4.);
\draw (8.,4.)-- (10.,4.);
\draw (12.,4.)-- (13.,4.);
\draw (14.,4.)-- (14.5,4.);
\draw (3.7556515807279486,4.056645955700061) node[anchor=north west] {$\frac 14$};
\draw (7.754529803796151,4.056645955700061) node[anchor=north west] {$\frac 12$};
\draw (9.779528589640193,4.056645955700061) node[anchor=north west] {$\frac 58$};
\draw (11.724154259277608,4.056645955700061) node[anchor=north west] {$\frac 34$};
\draw (12.690907419786374,4.056645955700061) node[anchor=north west] {$\frac{13}{16}$};
\draw (13.737660580295138,4.056645955700061) node[anchor=north west] {$\frac{7}{8}$};
\draw (14.20919011391112,4.056645955700061) node[anchor=north west] {$\frac{29}{32}$};
\draw (8.903894777751309,4.814854512245864) node[anchor=north west] {$\frac 47$};
\draw (2.059570515053495,4.814854512245864) node[anchor=north west] {$\frac 17$};
\draw (13.459489092538806,4.79966795414255) node[anchor=north west] {$\frac 67$};
\draw (11.176691935145057,4.7844813960392365) node[anchor=north west] {$\frac 57$};
\draw (12.209751002377019,4.830041070349178) node[anchor=north west] {$\frac{11}{14}$};
\draw (14.512921275977395,4.814854512245864) node[anchor=north west] {$\frac{13}{14}$};
\draw (0.23755665886247314,4.79966795414255) node[anchor=north west] {$\frac{1}{28}$};
\draw (2.535540374359535,4.79966795414255) node[anchor=north west] {$\frac{5}{28}$};
\draw (-0.25322664234025296,4.056645955700061) node[anchor=north west] {$0$};
\draw (15.7803032482345808,4.056645955700061) node[anchor=north west] {$1$};
\begin{scriptsize}
\draw [fill=qqqqff] (0.,4.) circle (1.0pt);
\draw [fill=qqqqff] (16.,4.) circle (1.0pt);
\draw [fill=ffqqqq] (9.142857142857142,4.) circle (2.5pt);
\draw [fill=ffqqqq] (2.2857142857142856,4.) circle (2.5pt);
\draw [fill=ffqqqq] (11.428571428571427,4.) circle (2.5pt);
\draw [fill=ffqqqq] (13.714285714285714,4.) circle (2.5pt);
\draw [fill=qqqqff] (12.571428571428571,4.) circle (2.5pt);
\draw [fill=qqqqff] (14.857142857142856,4.) circle (2.5pt);
\draw [color=qqqqff] (0.5714285714285714,4.)-- ++(-2.5pt,-2.5pt) -- ++(5.0pt,5.0pt) ++(-5.0pt,0) -- ++(5.0pt,-5.0pt);
\draw [color=qqqqff] (2.8571428571428568,4.)-- ++(-2.5pt,-2.5pt) -- ++(5.0pt,5.0pt) ++(-5.0pt,0) -- ++(5.0pt,-5.0pt);
\end{scriptsize}
\end{tikzpicture}
\caption{Optimal sets: of one-mean is $\set{\frac 47}$; of two-means is $\set{\frac 17, \frac 57}$; of three-means is $\set{\frac 17, \frac 47, \frac 67}$; of four-means is $\set{\frac 17, \frac 47, \frac{11}{14}, \frac{13}{14}}$; of five-means is $\set{\frac 1{28}, \frac 5{28}, \frac 47, \frac{11}{14}, \frac{13}{14}}$.}  \label{Fig1}
\end{figure}
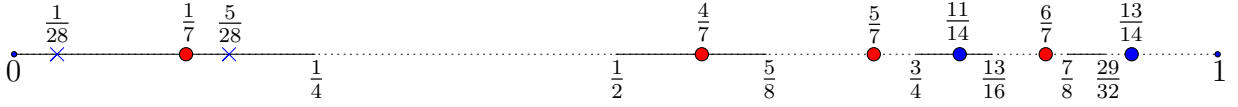

\begin{lemma} \label{lemma12}
Let $\ga$ be an optimal set of three-means. Then, $\ga=\set{a(1), a(2), a(2, \infty)}=\set{\frac 1{7}, \frac 4{7}, \frac 67}$ and the quantization error is $V_3=\frac{57}{14308}=0.00398379$.
\end{lemma}

\begin{proof} Let us first consider a three-point set $\gb$ given by $\gb:=\set{\frac 1{7}, \frac 4{7}, \frac 67}$. Since $J_{1}\sci M(\frac 1{7}|\gb)$, $J_{2}\sci M(\frac 4{7}|\gb)$ and $J_{(2, \infty)}\sci M(\frac 67|\gb)$, by Lemma~\ref{lemma4}, we have
\begin{align*}
&\int \min_{b \in \gb}(x-b)^2  dP=\int_{J_{1}} (x-\frac 1{7})^2 dP +\int_{J_2}(x-\frac 47)^2dP+\int_{J_{(2, \infty)}}(x-\frac 67)^2dP\\
&=p_1s_1^2 V +p_2s_2^2V(1+\frac{43}{9})=\frac{57}{14308}=0.00398379.
\end{align*}
Since $V_3$ is the quantization error for three-means, we have $V_3\leq\frac{57}{14308}=0.00398379$.
Let $\ga$ be an optimal set of three-means with $\ga=\set{a_1, a_2, a_3}$, where $a_1<a_2<a_3$. Since the optimal points are the centroids of their own Voronoi regions, we have  $0< a_1<a_2<a_3< 1$. If $a_1> \frac 14$, then
\[V_3\geq\int_{J_1}(x-a_1)^2 dP\geq \int_{J_1}(x-\frac{1}{4})^2 dP=\frac{135}{32704}= 0.00412794> V_3,\]
which gives a contradiction, and so $a_1\leq  \frac 14$. If $a_3< \frac{25}{32}=S_{32}(0)$, using \eqref{eq2}, we see that
\[V_3\geq \int_{J_{32}\uu J_{33}\uu \mathop{\uu}\limits_{j=4}^8 J_j }(x-\frac{25}{32})^2dP=\frac{8764935}{2143289344}=0.00408948>V_3,\]
which leads to a contradiction, and so $\frac{25}{32}\leq a_3$. Suppose that $a_2\leq \frac 12-\frac 1{32}=\frac{15}{32}$. Then, as $\frac 12(\frac{15}{32}+\frac{25}{32})=\frac 58=S_2(1)$, we have
\[V_3\geq \int_{J_2}(x-\frac{15}{32} )^2 dP=\frac{18525}{4186112}=0.00442535>V_3,\]
which is a contradiction. Assume that $\frac{15}{32}\leq a_2<\frac 12$. Then, $\frac 12(a_1+a_2)<\frac 14$ implying $a_1\leq \frac  12-a_2\leq \frac 12-\frac{15}{32}=\frac 1{32}<\frac 4{32}=S_{12}(0)$. Again $\frac 12(\frac 12+\frac{25}{32})=\frac {41}{32}>\frac 58=S_{2}(1)$. Thus, we have
\[V_3\geq \int_{J_{12}\uu J_{13}}(x-\frac 1{32})^2 dP+\int_{J_{2}}(x-\frac 12)^2dP=\frac{162087}{33488896}=0.00484002>V_3,\]
which is a contradiction. So, we can assume that $\frac 12\leq a_2$. Suppose that $\frac 58+\frac 1{32}=\frac {21}{32}\leq a_2$. Then, as $\frac 14<\frac 12 (a(1)+\frac{21}{32})<\frac 12$, we have
\[V_3\geq \int_{J_1}(x-a(1))^2 dP+\int_{J_2}(x-\frac {21}{32})^2 dP=\frac{129747}{29302784}=0.0044278>V_3,\]
which yields a contradiction. Next, suppose that $\frac 58< a_2\leq \frac 58+\frac 1{32}=\frac{21}{32}$. Then, $\frac 14<\frac 12 (a(1)+\frac{5}{8})<\frac 12$. Moreover, $\frac 12(a_2+a_3)>\frac 34$ implying $a_3>\frac 32-a_2\geq \frac 32-\frac {21}{32}=\frac {27}{32}>\frac{13}{16}=S_3(1)$ leading to the following two cases:

Case~A. $\frac {27}{32}< a_3\leq \frac{113}{128}=S_{41}(1)$.

Then, $\frac{1}{2} (\frac{21}{32}+\frac{27}{32})=\frac 34=S_{3}(0)$, and so
\begin{align*} V_3&\geq \int_{J_1}(x-a(1))^2 dP+\int_{J_2}(x-\frac 58)^2 dP+\int_{J_3}(x-\frac {27}{32})^2 dP+\int_{J_5\uu J_6\uu J_7}(x-\frac{113}{128})^2 dP\\
&=\frac{60087981}{15003025408}=0.00400506>V_3,
\end{align*}
which gives a contradiction.

Case~B. $S_{41}(1)=\frac{113}{128}\leq a_3$.

Then, $S_{31}(1)<\frac 12(\frac {21}{32}+\frac{113}{128})<S_{32}(0)$, and so
\begin{align*} V_3&\geq \int_{J_1}(x-a(1))^2 dP+\int_{J_2}(x-\frac 58)^2 dP+\int_{J_{31}}(x-\frac {21}{32})^2 dP+\int_{J_{32}\uu J_{33}}(x-\frac{113}{128})^2 dP\\
&=\frac{63174099}{15003025408}=0.00421076>V_3,
\end{align*}
which leads to a contradiction.

Therefore, $\frac 12\leq a_2\leq \frac 58$. Suppose that $S_{23}(0)=\frac {19}{32}\leq a_2\leq \frac 58$. Then, the Voronoi region of $a_2$ does not contain any point from $J_1$, and $\frac 12(a_2+a_3)>\frac 34$ implying $a_3>\frac 32-a_2\geq \frac 32-\frac 58=\frac 78$, otherwise the quantization error can strictly be reduced by moving the point $a_2$ to $a(2)=\frac 47$. Thus, we have
\begin{align*} &\min_{\frac {19}{32}\leq a_2\leq \frac 58} \int_{J_2}(x-a_2)^2dP=p_2\Big(s_2^2V+\min_{\frac {19}{32}\leq a_2\leq \frac 58}(S_2(\frac 47)-a_2)^2\Big)\\
&=p_2 \Big(s_2^2V+(a(2)-\frac{19}{32})^2\Big)=\frac{2757}{4186112}.
\end{align*}
The following two cases can arise:

Case~I. $\frac 78<a_3\leq S_{42}(0)=\frac{57}{64}$.

Then, $\frac{1}{2} (\frac 58+\frac 78)=\frac 34=S_{3}(0)$. Write $A:=J_{42}\uu J_{43}\uu \mathop{\uu}\limits_{j=5}^{10}J_j$, and so
\begin{align*} V_3&\geq \int_{J_1}(x-a(1))^2 dP+\min_{\frac {19}{32}\leq a_2\leq \frac 58} \int_{J_2}(x-a_2)^2dP+\int_{J_3}(x-\frac 78)^2dP +\int_A(x-\frac{57}{64})^2dP\\
&=\frac{3839362137}{960193626112}=0.00399853>V_3,
\end{align*}
which gives a contradiction.

Case~II. $S_{42}(0)=\frac{57}{64}<a_3$.

Then, $S_{311}(1)=\frac{193}{256}<\frac 12(\frac 58+\frac{57}{64})=\frac{97}{128}=S_{312}(0)$. Write $A:=\mathop{\uu}\limits_{j=2}^{10} J_{31j}\uu \mathop{\uu}\limits_{j=2}^{10} J_{3j}\uu J_{41}$. Thus,
\begin{align*} V_3&\geq \int_{J_1}(x-a(1))^2 dP+\min_{\frac {19}{32}\leq a_2\leq \frac 58} \int_{J_2}(x-a_2)^2dP+\int_{J_{311}}(x-\frac 58)^2dP+\int_A(x-\frac{57}{64})^2dP\\
& =\frac{1008051842887707}{251708997923504128}=0.00400483>V_3,
\end{align*}
which leads to a contradiction.

Therefore, we can assume that $\frac 12\leq a_2\leq \frac {19}{32}=S_{23}(0)$. Again, we have seen that $\frac {25}{32}\leq a_3\leq 1$. Then, notice that the Voronoi region of $a_2$ does not contain any point from $J_1$. Moreover, $\frac {41}{64}=\frac 12(\frac 12+\frac {25}{32})\leq \frac 12(a_2+a_3)\leq \frac 12(\frac{19}{32}+1)=\frac {51}{64}$ implying that the Voronoi region of $a_3$ does not contain any point from $J_2$. Suppose that the Voronoi region of $a_2$ contains points from $J_{(2, \infty)}$. Then, $\frac 12(a_2+a_3)>\frac 34$, which implies $a_3>\frac 32-a_2\geq \frac 32-\frac {19}{32}=\frac{29}{32}=S_{4}(1)$. Moreover,
\[\min_{\frac 12\leq a_2\leq \frac {19}{32}} \int_{J_2}(x-a_2)^2 dP=\int_{J_2}(x-a(2))^2 dP=p_2s_2^2V.\]
Thus, we see that
\begin{align*} V_3&\geq \int_{J_1}(x-a(1))^2 dP+p_2s_2^2V+\int_{J_3\uu J_4}(x-\frac{29}{32})^2dP=\frac{531801}{117211136}=0.00453712>V_3,
\end{align*}
which gives a contradiction. Therefore, we can assume that the Voronoi region of $a_2$ does not contain any point from $J_{(2,\infty)}$. Thus, we have proved that $J_1\sci M(a_1|\ga)$, $J_2\sci M(a_2|\ga)$, and $J_3\sci M(a_3|\ga)$ yielding $a_1=a(1)=\frac 17$, $a_2=a(2)=\frac 47$, and $a_3=a(2, \infty)=\frac 67$, and the corresponding quantization error is $V_3=\frac{57}{14308}=0.00398379$  (see Figure~\ref{Fig1}). Thus, the proof of the lemma is complete.
\end{proof}

We need the following two lemmas to prove Proposition~\ref{prop1}.

\begin{lemma} \label{lemma13}
Let $\ga_4$ be an optimal set of four-means. Then, $\ga_4\ii J_1\neq \es$ and $\ga_4\ii J_{(1, \infty)}\neq \es$, and $\ga_4$ does not contain any point from the open interval $(\frac 14, \frac 12)$. Moreover, the Voronoi region of any point in $\ga_4\ii J_1$ does not contain any point from $J_{(1, \infty)}$ and the Voronoi region of any point in $\ga_4\ii J_{(1, \infty)}$ does not contain any point from $J_1$.
\end{lemma}

\begin{proof} Let $\ga_4:=\set{0<a_1<a_2<a_3<a_4<1}$ be an optimal set of four-means. Consider the set $\gb:=\set{a(1), a(2), a(3), a(3, \infty)}$ of four points. Then,
\[\int \min_{a\in \gb} (x-a)^2 dP=p_1s_1^2V+p_2s_2^2V+p_3s_3^3V(1+\frac{43}{9})=\frac{237}{114464}=0.00207052.\]
Since $V_4$ is the quantization error for four-means, we have $V_4\leq 0.00207052$. If $a_1\geq \frac{13}{64}=S_{13}(1)$, we have
\[V_4\geq \int_{J_{11}\uu J_{12}}(x-\frac{13}{64})^2 dP=\frac{20277}{9568256}=0.00211919>V_4,\]
which is a contradiction. So, we can assume that $a_1\leq\frac{13}{64}$. Then, the Voronoi region of $a_1$ does not contain any point from $J_{(1, \infty)}$. If it does, then $\frac 12(a_1+a_2)>\frac 12$ implies $a_2\geq 1-a_1\geq 1-\frac{13}{64}=\frac {51}{64}$ which is a contradiction as
\[V_4\geq \int_{J_1}(x-a(1))^2 dP+\int_{J_2}(x-\frac {51}{64})dP=\frac{2436771}{117211136}=0.0207896>V_4.\]
If $a_4\leq \frac {53}{64}$, then
\[V_4\geq \int_{\mathop{\uu}\limits_{j=4}^{10}J_j}(x-\frac {53}{64})^2 dP=\frac{292246431}{137170518016}=0.00213053>V_4,\]
which is a contradiction, and so $\frac {53}{64}<a_4$.
If $a_2\leq \frac 14$, then
\[V_4\geq \int_{J_2}(x-a(2))^2dP+\int_{J_{(2, \infty)}}(x-a(2, \infty))^2 dP=(1+\frac{43}{9})p_2s_2^2 V=\frac{39}{14308}=0.00272575>V_4,\]
which gives a contradiction. So, we can assume that $\frac 14< a_2$. Suppose that $\frac 14 <a_2\leq \frac 38$. Then, $\frac 12(a_2+a_3)>\frac 12$ yielding
$a_3>1-a_2\geq 1-\frac 38=\frac 58$. Thus, the following two cases can arise:

Case~1. $\frac 58< a_3\leq \frac{43}{64}$.

Then, as $\frac {53}{64}<a_4$ and $\frac 12(\frac {43}{64}+\frac{53}{64})=\frac 34$, we have
\[V_4\geq \int_{J_2}(x-\frac 58)^2 dP+\int_{J_3}(x-\frac{53}{64})^2 dP+\int_{J_{(3, \infty)}}(x-a(3, \infty))^2 dP=\frac{521811}{234422272}=0.00222594>V_4,\]
which is a contradiction.

Case~2. $\frac{43}{64}\leq a_3$.

Then, as $S_{212}(1)<\frac 12(\frac 38+\frac {43}{64})=\frac{67}{128}=S_{213}(0)$, we have
\[V_4\geq \int_{J_{211}\uu J_{212}}(x-\frac 38)^2 dP+\int_{J_{22}\uu J_{23}}(x-\frac{43}{64})^2 dP=\frac{6099}{2093056}=0.00291392>V_4,\]
which leads to a contradiction.

Thus, a contradiction arises to our assumption $\frac 14 <a_2\leq \frac 38$. Suppose that $\frac 38\leq a_2<\frac 12$.
Then, $\frac 12(a_1+a_2)<\frac 1 4$ implying $a_1\leq \frac 12-a_2\leq \frac 12-\frac 38=\frac 18<a(1)$, and
\begin{equation*}  \min_{\set{a_1< \frac 18<\frac 38\leq a_2}} \int_{J_1}\min_{a\in \set{a_1, a_2}}(x-a)^2 dP\geq \int_{J_1}(x-a(1))^2 dP=\frac{9}{7154}.\end{equation*}
Since $\frac {53}{64}<a_4$, the following three cases can arise:

Case~I. $a_3\leq \frac{43}{64}$ and $\frac {53}{64}<a_4\leq \frac 78$.

Then, as $\frac 12(\frac {43}{64}+\frac{53}{64})=\frac 34$, we have
\[V_4\geq \int_{J_1}(x-a(1))^2 dP+\int_{J_3}(x-\frac {53}{64})^2 dP+\int_{J_4\uu J_5\uu J_6}(x-\frac 78)^2 dP=\frac{126459}{58605568}=0.0021578>V_4,\]
which gives a contradiction.

Case~II. $a_3\leq \frac{43}{64}$ and $\frac 78\leq a_4 $.

Then, as $S_{31}(1)<\frac 12(\frac {43}{64}+\frac{7}{8})<S_{32}(0)$,
\[V_4\geq \int_{J_1}(x-a(1))^2 dP+\int_{J_{31}}(x-\frac {43}{64})^2 dP+\int_{J_{32}\uu J_{33}}(x-\frac 78)^2 dP=\frac{4458897}{1875378176}=0.0023776>V_4,\]
which leads to a contradiction.

Case~III. $\frac{43}{64}\leq a_3$.

Then, $S_{22}(1)<\frac 12(\frac 12+\frac{43}{64})<S_{23}(0)$ yielding
\[V_4\geq \int_{J_1}(x-a(1))^2 dP+\int_{J_{21}\uu J_{22}}(x-\frac 12)^2 dP+\int_{J_{23}}(x-\frac {43}{64})^2dP=\frac{4496025}{1875378176}=0.0023974>V_4,\]
which is a contradiction.

Thus, a contradiction arises to our assumption $\frac 38\leq a_2<\frac 12$, and so we can assume $\frac 12\leq a_2$. Now, notice that $\frac 12(a_1+a_2)\geq \frac 12(0+\frac 12)=\frac 14$ yielding the fact that the Voronoi region of any point in $\ga_4\ii J_{(1, \infty)}$ does not contain any point from $J_1$. Moreover, we proved $a_1<\frac 14$ and the Voronoi region of any point in $\ga_4\ii J_1$ does not contain any point from $J_{(1, \infty)}$. Thus, the proof of the lemma is complete.
\end{proof}

\begin{lemma} \label{lemma14}
Let $\ga_5$ be an optimal set of five-means. Then, $\ga_5\ii J_1\neq \es$, $\ga_5\ii J_{(1, \infty)}\neq \es$, and $\ga_5$ does not contain any point from the open interval $(\frac 14, \frac 12)$. Moreover, the Voronoi region of any point in $\ga_5\ii J_1$ does not contain any point from $J_{(1, \infty)}$ and the Voronoi region of any point in $\ga_5\ii J_{(1, \infty)}$ does not contain any point from $J_1$.
\end{lemma}

\begin{proof} Let $\ga_5:=\set{0<a_1<a_2<a_3<a_4<a_5<1}$ be an optimal set of five-means. Consider the set $\gb:=\set{a(11), a(11, \infty), a(2), a(3), a(3, \infty)}$ of five points. Then,
\[\int \min_{a\in \gb} (x-a)^2 dP=p_{11}s_{11}^2V(1+\frac{43}{3}) +p_2s_2^2V+p_3s_3^3V(1+\frac{43}{9})=\frac{255}{228928}=0.00111389.\]
Since $V_5$ is the quantization error for five-means, we have $V_5\leq 0.00111389$. If $a_5\leq \frac 6 7$, then
\[V_5\geq \int_{\mathop{\uu}\limits_{j=4}^{10}J_j}(x-\frac 67)^2 dP=\frac{1160604105}{960193626112}=0.00120872>V_5,\]
which is a contradiction, and so $\frac 67<a_5$. Suppose that $a_4\leq \frac{11}{16}$. Consider the following two cases:

Case~1. $\frac 67\leq a_5<\frac 78$.

Then, $S_{31}(1)<\frac 12(\frac{11}{16}+\frac 67)<\frac{25}{32}=S_{32}(0)$, yielding
\[V_5\geq \int_{J_{31}}(x-\frac {11}{16})^2 dP+\int_{J_{32}\uu J_{33}}(x-\frac 67)^2 dP+\int_{\mathop{\uu}\limits_{j=4}^6J_j}(x-\frac 78)^2 dP=\frac{2290131}{1875378176}=0.00122116>V_5,\]
which leads to a contradiction.

Case~2. $\frac 78\leq a_5$.

Then, $S_{31}(1)<\frac 12(\frac{11}{16}+\frac 78)=\frac{25}{32}=S_{32}(0)$, yielding
\[V_5\geq \int_{J_{31}}(x-\frac {11}{16})^2 dP+\int_{\mathop{\uu}\limits_{j=2}^{10}J_{3j}}(x-\frac 78)^2 dP=\frac{651896011533}{561850441793536}=0.00116027>V_5,\]
which is a contradiction.

Hence, we can assume that $\frac {11}{16}<a_4$. If $a_3\leq \frac 14$, then
\[V_5\geq \int_{J_2}(x-a(2))^2dP+\int_{J_{(2, \infty)}}(x-a(2, \infty))^2 dP=(1+\frac{43}{9})p_2s_2^2 V=\frac{39}{14308}=0.00272575>V_5,\]
which gives a contradiction. So, we can assume that $\frac 14< a_3$. Suppose that $\frac 14 <a_3< \frac 12$. The following two cases can arise:

Case~(i). $\frac 14< a_3\leq \frac 38$.

Then, $\frac 12(a_3+a_4)>\frac 12$ implying $a_4>1-a_3\geq 1-\frac 38=\frac 58$, and so
\[V_5\geq \int_{J_2}(x-\frac 58)^2 dP=\frac{405}{261632}=0.00154798>V_5,\]
which is a contradiction.

Case~(ii). $\frac 38 \leq a_3<\frac 12$.

Then, $\frac 12(a_2+a_3)<\frac 14$ implying $a_2<1-a_3\leq \frac 12-\frac 38=\frac 18$. Moreover, as $\frac {11}{16}<a_4$, we have
$S_{22}(1)<\frac 12(\frac 12+\frac {11}{16})=\frac{19}{32}=S_{23}(0)$, and so
\[V_5\geq \int_{J_{12}}(x-\frac {1}{8})^2 dP+\int_{J_{21}\uu J_{22}}(x-\frac 12)^2 dP+\int_{J_{23}}(x-\frac{11}{16})^2 dP=\frac{45399}{33488896}=0.00135564>V_5,\]
which yields a contradiction.

Hence, we can assume that $\frac 12\leq  a_3$. If $\frac 38\leq a_2$, then
\[V_5\geq \min_{\set{a_1< \frac 18<\frac 38\leq a_2}} \int_{J_1}\min_{a\in \set{a_1, a_2}}(x-a)^2 dP\geq \int_{J_1}(x-a(1))^2 dP=\frac{9}{7154}=0.00125804>V_5,\]
which is a contradiction. Suppose that $\frac 14<a_2\leq \frac 38$. Then, $\frac 12(a_2+a_3)>\frac 12$ implying $a_3>1-a_2\geq 1-\frac 38=\frac 58$, which yields \[V_5\geq \int_{J_2}(x-\frac 58)^2 dP=\frac{405}{261632}=0.00154798>V_5,\]
leading to a contradiction. So, we can assume that $a_2\leq \frac 14$. Thus, we have proved that $a_2\leq \frac 14$ and $\frac 12\leq a_3$, yielding the fact that $\ga_5\ii J_1\neq \es$, $\ga_5\ii J_{(1, \infty)}\neq \es$, and $\ga_5$ does not contain any point from the open interval $(\frac 14, \frac 12)$. Since $\frac 12(a_2+a_3)\geq \frac 12(0+\frac 12)=\frac 14$, the Voronoi region of any point in $\ga_5\ii J_{(1, \infty)}$ does not contain any point from $J_1$. If the Voronoi region of $a_2$ contains points from $J_{(1, \infty)}$, then $\frac 12(a_2+a_3)>\frac 12$ implying $a_3>1-a_2\geq 1-\frac 14=\frac 34$, and so
\[V_5\geq \int_{J_2}(x-\frac 34)^2 dP=\frac{813}{65408}=0.0124297>V_5,\]
which gives a contradiction. Thus, the proof of the lemma is complete.
\end{proof}

\begin{prop} \label{prop1}
Let $\ga_n$ be an optimal set of $n$-means for $n\geq 2$. Then, $\ga_n\ii J_1\neq \es$ and $\ga_n\ii J_{(1, \infty)}\neq \es$,  and $\ga_n$ does not contain any point from the open interval $(\frac 14, \frac 12)$. Moreover, the Voronoi region of any point in $\ga_n\ii J_1$ does not contain any point from $J_{(1, \infty)}$ and the Voronoi region of any point in $\ga_n\ii J_{(1, \infty)}$ does not contain any point from $J_1$.
\end{prop}

\begin{proof} By Lemma~\ref{lemma11}, Lemma~\ref{lemma12}, Lemma~\ref{lemma13}, and Lemma~\ref{lemma14},  the proposition is true for $2\leq n\leq 5$. We now prove the proposition for all $n\geq 6$. Let $\ga_n:=\set{0<a_1<a_2<\cdots <a_n<1}$ be an optimal set of $n$-means for $n\geq 6$. Consider the set of six points $\gb:=\set{a(11), a(11, \infty), a(21), a(21, \infty), a(3), a(3, \infty)}$. Then, the distortion error is
\begin{align*}
\int\min_{a\in\gb}(x-a)^2 dP=(1+\frac{43}{3})p_{11}s_{11}^2 V+(1+\frac{43}{3})p_{21}s_{21}^2 V+(1+\frac{43}{9})p_{3}s_{3}^2 V=\frac{1383}{1831424}.
\end{align*}
Since, $V_n$ is the quantization error for $n$-means for $n\geq 6$, we have $V_n\leq V_6\leq \frac{1383}{1831424}=0.00075515$. Proceeding in the similar way, as shown in the previous lemmas, we have $a_1<\frac 14$ and $\frac 12<a_n$. Let $j=\max\set{i : a_i<\frac 12}$. Then, $a_j<\frac 12$. We show that $a_j\leq \frac 14$. Suppose that $\frac 14<a_j<\frac 12$. Then, the following two cases can arise:

Case~1. $\frac 38\leq a_j<\frac 12$.

Then, $\frac 12(a_{j-1}+a_j)<\frac 14$ implying $a_{j-1}<\frac 12-a_j\leq \frac 12-\frac 38=\frac 18=S_{12}(0)$ yielding
\[V_n\geq \int_{\mathop{\uu}\limits_{j=2}^{10}J_{1j}}(x-\frac 18)^2 dP=\frac{13986897}{17179869184}=0.000814145>V_n,\]
which is a contradiction.

Case~2. $\frac 14<a_j\leq \frac 38$.

Then, $\frac 12(a_j+a_{j+1})>\frac 12$ implying $a_{j+1}>1-a_j\geq 1-\frac 38=\frac 58$ yielding
\[V_n\geq \int_{J_2}(x-\frac 58)^2 dP=\frac{405}{261632}=0.00154798>V_n,\]
which gives a contradiction.

Hence, we can assume that $a_j\leq \frac 12$. Thus, we have seen that $\ga_n\ii J_1\neq \es$, $\ga_n\ii J_{(1, \infty)}\neq \es$, and $\ga_n$ does not contain any point from the open interval $(\frac 14, \frac 12)$. Since $\frac 12(a_j+a_{j+1})\geq \frac 12(0+\frac 12)=\frac 14$, the Voronoi region of any point in $\ga_n\ii J_{(1, \infty)}$ does not contain any point from $J_1$. Suppose that the Voronoi region of $a_j$ contains points from $J_{(1, \infty)}$. Then, $\frac 12(a_j+a_{j+1})>\frac 12$ implying $a_{j+1}>1-a_2\geq 1-\frac 14=\frac 34$, and so
\[V_n\geq \int_{J_2}(x-\frac 34)^2 dP=\frac{813}{65408}=0.0124297>V_n,\]
which is a contradiction. So, we can assume that the Voronoi region of any point in $\ga_n\ii J_1$ does not contain any point from $J_{(1, \infty)}$
Thus, the proof of the proposition is complete.
\end{proof}

We need the following lemmas to prove Proposition~\ref{prop2}.

\begin{lemma} \label{lemma141}
Let $V(P, J_2, \set{a, b})$ be the quantization error due to the points $a$ and $b$ on the set $J_2$, where $\frac 12\leq a<b$ and $b=\frac 58$. Then, $a=a(21, 22)$ and
\[V(P, J_2, \set{a, b})=\int_{J_{21}\uu J_{22}}(x-a(21,22))^2 dP+\int_{J_{(22, \infty)}}(x-\frac 58)^2dP=\frac{2403}{10465280}.\]
\end{lemma}

\begin{proof}  Consider the set $\set{\frac {11}{20}, \frac 58}$. Then, as $S_{22}(1)<\frac 12(\frac{11}{20}+\frac 58)<S_{23}(0)$, and $V(P, J_2, \set{a, b})$ is the quantization error due to the points $a$ and $b$ on the set $J_2$, we have
\[V(P, J_2, \set{a, b})\leq \int_{J_{21}\uu J_{22}}(x-\frac{11}{20})^2 dP+\int_{J_{(22, \infty)}}(x-\frac 58)^2dP=\frac{2403}{10465280}=0.000229616. \]
If $\frac {37}{64}=S_{22}(1)\leq a$, then
\[V(P, J_2, \set{a, b})\geq \int_{J_{21}\uu J_{22}}(x-S_{22}(1))^2 dP=\frac{6831}{19136512}=0.000356962>V(P, J_2, \set{a, b}),\]
which is a contradiction, and so we can assume that $a<S_{22}(1)=\frac{37}{64}$. If the Voronoi region of $b$ contains points from $J_{22}$, we must have
$\frac 12(a+b)<\frac{37}{64}$ implying $a<\frac {37}{32}-b= \frac{37}{32}-\frac 58=\frac{17}{32}=S_{21}(1)$, and so
\[V(P, J_2, \set{a, b})>\int_{J_{22}}(x-\frac{17}{32})^2 dP+\int_{\mathop{\uu}\limits_{j=3}^{10}J_{2j}}(x-\frac 58)^2dP=\frac{276910245}{962072674304}=0.000287827,\]
yielding $V(P, J_2, \set{a, b})>0.000287827>V(P, J_2, \set{a, b})$, which leads to a contradiction. So, we can assume that the Voronoi region of $b$ does not contain any point from $J_{22}$ yielding $a\geq a(21, 22)=\frac{11}{20}$. If the Voronoi region of $a$ contains points from $J_{23}$, we must have
$\frac 12(a+\frac 58)>S_{23}(0)=\frac {19}{32}$ implying $a>\frac {19}{16}-\frac 58=\frac 9{16}=S_{22}(0)$, and then
\[V(P, J_2, \set{a, b})>\int_{J_{21}}(x-\frac{9}{16})^2 dP+\int_{\mathop{\uu}\limits_{j=3}^{10}J_{2j}}(x-\frac 58)^2dP=\frac{17716739853}{70231305224192}=0.000252263,\]
yielding $V(P, J_2, \set{a, b})>0.000252263>V(P, J_2, \set{a, b})$, which leads to a contradiction. So, the Voronoi region of $a$ does not contain any point from $J_{23}$ yielding $a\leq a(21, 22)$. Again, we proved $a\geq a(21, 22)$. Thus, $a=a(21, 22)$ and
\[V(P, J_2, \set{a, b})=\int_{J_{21}\uu J_{22}}(x-a(21,22))^2 dP+\int_{J_{(22, \infty)}}(x-\frac 58)^2dP=\frac{2403}{10465280}.\]
Thus, the proof of the lemma is complete.
\end{proof}

\begin{lemma} \label{lemma15}
Let $\ga_6$ be an optimal set of six-means. Then, $\te{card} (\ga_6\ii J_1)=2$ and $\te{card} (\ga_6\ii J_{(1, \infty)})=4$. Moreover, $\te{card} (\ga_6\ii J_2)=2$.
\end{lemma}
\begin{proof}
Let $\ga_6:=\set{0<a_1<a_2<a_3<a_4<a_5<a_6<1}$ be an optimal set of six-means. Consider the set of six points $\gb:=\set{a(11), a(11, \infty), a(21), a(21, \infty), a(3), a(3, \infty)}$. Then, the distortion error is
\begin{align*}
\int\min_{a\in\gb}(x-a)^2 dP=(1+\frac{43}{3})p_{11}s_{11}^2 V+(1+\frac{43}{3})p_{21}s_{21}^2 V+(1+\frac{43}{9})p_{3}s_{3}^2 V=\frac{1383}{1831424}.
\end{align*}
Since, $V_6$ is the quantization error for six-means, we have $V_6\leq \frac{1383}{1831424}=0.00075515$. By Proposition~\ref{prop1}, we have $\te{card}(\ga_6\ii J_1)\geq 1$ and $\te{card}(\ga_6\ii J_{(1, \infty)})\geq 1$. Moreover, the Voronoi region of any point in $\ga_6\ii J_1$ does not contain any point from $J_{(1, \infty)}$ and the Voronoi region of any point in $\ga_6\ii J_{(1, \infty)}$ does not contain any point from $J_1$.  Suppose that $\te{card} (\ga_6\ii J_{(1, \infty)})=2$, and then taking $\gb_2=\set{a(2), a(2, \infty)}$ we see that
\[V_6\geq \int_{J_2\uu J_{(2, \infty)}} \min_{a\in \gb_2}(x-a)^2 dP= \int_{J_2}(x-a(2))^2 dP+\int_{J_{(2, \infty)}}(x-a(2, \infty))^2 dP=\frac{39}{14308}=0.00272575,\]
i.e., $V_6\geq 0.00272575>V_6$, which yields a contradiction. Next, assume that $\te{card} (\ga_6\ii J_{(1, \infty)})=3$, and then taking $\gb_2=\set{a(2), a(3), a(3, \infty)}$, we see that
\[V_6\geq \int_{J_2}(x-a(2))^2 dP+\int_{J_3}(x-a(3))^2 dP+\int_{J_{(3, \infty)}}(x-a(3, \infty))^2 dP=\frac{93}{114464}=0.000812483>V_6,\]
which gives a contradiction. Thus, we can assume that  $\te{card} (\ga_6\ii J_{(1, \infty)})\geq 4$. If $\te{card} (\ga_6\ii J_1)=1$, then,
\[V_6\geq \int_{J_1}(x-a(1))^2 dP=\frac{9}{7154}=0.00125804>V_6,\]
which yields a contradiction, and so $\te{card} (\ga_6\ii J_1)\geq 2$. Therefore, we can assume that $\te{card} (\ga_6\ii J_1)=2$ and $\te{card} (\ga_6\ii J_{(1, \infty)})=4$. We now show that $\te{card} (\ga_6\ii J_2)=2$. By Proposition~\ref{prop1}, the Voronoi region of any element in $\ga_6\ii J_1$ does not contain any point from $J_{(1, \infty)}$, and the Voronoi region of any element in $\ga_6\ii J_{(1, \infty)}$ does not contain any point from $J_1$. We have $\ga_6\ii J_{(1, \infty)}=\set{\frac 12\leq a_3<a_4< a_5< a_6<1}$. The distortion error contributed by the set $\gb\ii J_{(1, \infty)}=\set{a(21), a(21, \infty), a(3), a(3, \infty)}$ is given by
\[\int_{J_{(1, \infty)}} \min_{a\in \gb\ii J_{(1, \infty)}}(x-a)^2 dP=(1+\frac {43}{3})p_{21}s_{21}^2V+(1+\frac {43}{9})p_3s_3^2V=\frac{831}{1831424}=0.000453745.\]
Let $V(P, \ga_6\ii J_{(1, \infty)})$ be the quantization error contributed by the set $\ga_6\ii J_{(1, \infty)}$ in the region $J_{(1, \infty)}$. Then, we must have
 $V(P, \ga_6\ii J_{(1, \infty)})\leq 0.000453745$.
 If $a_6\leq \frac{57}{64}=S_{42}(0)$, then
\[V(P, \ga_6\ii J_{(1, \infty)})\geq \int_{\mathop{\uu}\limits_{j=5}^{8}J_j}(x-\frac{57}{64})^2 dP=\frac{145935}{306184192}=0.000476625>V(P, \ga_6\ii J_{(1, \infty)}),\]
 which yields a contradiction, and so $S_{42}(0)=\frac{57}{64}< a_6$.
 If $\frac 34<a_4$, then
 \[V(P, \ga_6\ii J_{(1, \infty)}) \geq \int_{J_2}(x-a(2))^2 dP=\frac{27}{57232}=0.000471764>V(P, \ga_6\ii J_{(1, \infty)}),\]
 which yields a contradiction. So, we can assume that $a_4<\frac 34$. Suppose that $\frac 58<a_4<\frac 34$. Then, the following two cases can arise:

 Case~1. $\frac {11}{16}\leq a_4<\frac 34$.

 Then, $\frac 12(a_3+a_4)<\frac 58$ implying $a_3<\frac 54-a_4\leq \frac 54-\frac{11}{16}=\frac 9{16}$, and so
\[V(P, \ga_6\ii J_{(1, \infty)}) \geq  \min_{\set{a_3< \frac 9{16}<\frac {11}{16}\leq a_4}} \int_{J_2}\min_{a\in \set{a_3, a_4}}(x-a)^2 dP\geq \int_{J_2}(x-a(2))^2 dP=\frac{27}{57232}, \]
implying $V(P, \ga_6\ii J_{(1, \infty)}) \geq\frac{27}{57232}= 0.000471764>V(P, \ga_6\ii J_{(1, \infty)})$, which gives a contradiction.

 Case~2. $\frac 58 <a_4<\frac {11}{16}$.

 Then, $\frac 12(a_4+a_5)>\frac 34$ implying $a_5>\frac 32-a_4\geq \frac 32-\frac{11}{16}=\frac{13}{16}$. Then, the following two subcases can arise:

 Subcase~(i). $\frac {27}{32}\leq a_5$.

 Then, $S_{31}(1)=\frac{49}{64}=\frac 12(\frac {11}{16}+\frac {27}{32})<S_{32}(0)$, and so by Lemma~\ref{lemma141},
\begin{align*}& V(P, \ga_6\ii J_{(1, \infty)})\\
 &\geq \int_{J_{21}\uu J_{22}} (x-a(21, 22))^2 dP+ \int_{J_{(22, \infty)}}(x-\frac 58)^2 dP+\int_{J_{31}}(x-\frac {11}{16})^2dP +\int_{J_{32}}(x-\frac {27}{32})^2 dP\\
 &=\frac{236721}{334888960}=0.000706864>V(P, \ga_6\ii J_{(1, \infty)}),
\end{align*}
 which gives a contradiction.

Subcase~(ii). $\frac {13}{16}<a_5<\frac {27}{32}$.

 Then, $\frac 12(a_5+a_6)>\frac 78$ implying $a_6>\frac 74-a_5\geq \frac 74-\frac {27}{32}=\frac {29}{32}=S_4(1)$. First, assume that $S_4(1)<a_6<S_5(0)=\frac{15}{16}$. Then, using Lemma~\ref{lemma141},
 \begin{align*}& V(P, \ga_6\ii J_{(1, \infty)})\\
 &\geq \int_{J_{21}\uu J_{22}} (x-a(21, 22))^2 dP+ \int_{J_{(22, \infty)}}(x-\frac 58)^2 dP+\int_{J_{3}}(x-\frac {13}{16})^2dP +\int_{J_4}(x-\frac {29}{32})^2 dP\\
 &+\int_{J_5\uu J_6}(x-\frac{15}{16})^2dP=\frac{11529}{23920640}=0.000481969>V(P, \ga_6\ii J_{(1, \infty)}),
\end{align*}
 which leads to a contradiction. Next, assume that $S_5(0)=\frac {15}{16}\leq a_6$. Then, as $S_{42}(0)=\frac{57}{64}=\frac 12(\frac{27}{32}+\frac {15}{16})$, using Lemma~\ref{lemma141},  we have
  \begin{align*}& V(P, \ga_6\ii J_{(1, \infty)})\\
 &\geq \int_{J_{21}\uu J_{22}} (x-a(21, 22))^2 dP+ \int_{J_{(22, \infty)}}(x-\frac 58)^2 dP+\int_{J_{3}}(x-\frac {13}{16})^2dP +\int_{J_{41}}(x-\frac {27}{32})^2 dP\\
 &+\int_{J_{42}}(x-\frac{15}{16})^2dP=\frac{700899}{1339555840}=0.000523232>V(P, \ga_6\ii J_{(1, \infty)}),
\end{align*}
 which yields a contradiction.

Hence, by Case~1 and Case~2, we can assume that $a_4\leq \frac 58$ yielding $\te{card} (\ga_6\ii J_2)=2$. Thus, the proof of the proposition is complete.
\end{proof}

\begin{lemma} \label{lemma153}
Let $\ga_7$ be an optimal set of seven-means. Then, either $(i)$ $\te{card} (\ga_7\ii J_1)=3$ and $\te{card} (\ga_7\ii J_{(1, \infty)})=4$, or $(ii)$ $\te{card} (\ga_7\ii J_1)=2$ and $\te{card} (\ga_7\ii J_{(1, \infty)})=5$.
\end{lemma}

\begin{proof}
Let $\ga_7:=\set{0<a_1<a_2<\cdots<a_7<1}$ be an optimal set of seven-means. Consider the set of seven points $\gb:=\set{a(11), a(12), a(12, \infty), a(21), a(21, \infty), a(3), a(3, \infty)}$. Then, the distortion error due to the set $\gb$ is
\begin{align*}
\int\min_{a\in\gb}(x-a)^2 dP=p_{11}s_{11}^2V+ (1+\frac{43}{9})p_{12}s_{12}^2 V+(1+\frac{43}{3})p_{21}s_{21}^2 V+(1+\frac{43}{9})p_{3}s_{3}^2 V=\frac{135}{261632}.
\end{align*}
Since, $V_7$ is the quantization error for seven-means, we have $V_7\leq \frac{135}{261632}=0.000515992$. By Proposition~\ref{prop1}, we have $\te{card}(\ga_7\ii J_1)\geq 1$ and $\te{card}(\ga_7\ii J_{(1, \infty)})\geq 1$. Moreover, the Voronoi region of any point in $\ga_7\ii J_1$ does not contain any point from $J_{(1, \infty)}$ and the Voronoi region of any point in $\ga_7\ii J_{(1, \infty)}$ does not contain any point from $J_1$. Suppose that $\te{card} (\ga_7\ii J_{(1, \infty)})=2$, and then taking $\gb_2=\set{a(2), a(2, \infty)}$ we see that
\[V_7\geq \int_{J_2\uu J_{(2, \infty)}} \min_{a\in \gb_2}(x-a)^2 dP= \int_{J_2}(x-a(2))^2 dP+\int_{J_{(2, \infty)}}(x-a(2, \infty))^2 dP=\frac{39}{14308}=0.00272575,\]
i.e., $V_7\geq 0.00272575>V_7$, which yields a contradiction. Next, assume that $\te{card} (\ga_7\ii J_{(1, \infty)})=3$, and then taking $\gb_2=\set{a(2), a(3), a(3, \infty)}$, we see that
\[V_7\geq \int_{J_2}(x-a(2))^2 dP+\int_{J_3}(x-a(3))^2 dP+\int_{J_{(3, \infty)}}(x-a(3, \infty))^2 dP=\frac{93}{114464}=0.000812483>V_7,\]
which gives a contradiction. Thus, we can assume that  $\te{card} (\ga_7\ii J_{(1, \infty)})\geq 4$. If $\te{card} (\ga_7\ii J_1)=1$, then,
\[V_7\geq \int_{J_1}(x-a(1))^2 dP=\frac{9}{7154}=0.00125804>V_7,\]
which gives a contradiction. So, we can assume that $\te{card} (\ga_7\ii J_1)\geq 2$. Thus, we have either $(i)$ $\te{card} (\ga_7\ii J_1)=3$ and $\te{card} (\ga_7\ii J_{(1, \infty)})=4$, or $(ii)$ $\te{card} (\ga_7\ii J_1)=2$ and $\te{card} (\ga_7\ii J_{(1, \infty)})=5$, which is the lemma.
\end{proof}

\begin{lemma} \label{lemma1531}
Let $\ga_8$ be an optimal set of eight-means. Then, $\te{card} (\ga_8\ii J_1)=3$ and $\te{card} (\ga_8\ii J_{(1, \infty)})=5$.
\end{lemma}

\begin{proof} Let $\ga_8:=\set{0<a_1<a_2<\cdots<a_8<1}$ be an optimal set of eight-means. Consider the set of eight points $\gb:=\set{a(11), a(12), a(12, \infty), a(21), a(21, \infty), a(3), a(4), a(4, \infty)}$. Then, the distortion error due to the set $\gb$ is
\begin{align*}
&\int\min_{a\in\gb}(x-a)^2 dP\\
&=p_{11}s_{11}^2V+ (1+\frac{43}{9})p_{12}s_{12}^2 V+(1+\frac{43}{3})p_{21}s_{21}^2 V+p_3s_3^2V+(1+\frac{43}{9})p_{4}s_{4}^2 V=\frac{507}{1831424}.
\end{align*}
Since $V_8$ is the quantization error for eight-means, we have $V_8\leq \frac{507}{1831424}=0.000276834$. By Proposition~\ref{prop1}, we have $\te{card}(\ga_8\ii J_1)\geq 1$ and $\te{card}(\ga_8\ii J_{(1, \infty)})\geq 1$. Moreover, the Voronoi region of any point in $\ga_8\ii J_1$ does not contain any point from $J_{(1, \infty)}$ and the Voronoi region of any point in $\ga_8\ii J_{(1, \infty)}$ does not contain any point from $J_1$. Suppose that $\te{card} (\ga_8\ii J_{(1, \infty)})=2$, and then taking $\gb_2=\set{a(2), a(2, \infty)}$ we see that
\[V_8\geq \int_{J_2\uu J_{(2, \infty)}} \min_{a\in \gb_2}(x-a)^2 dP= \int_{J_2}(x-a(2))^2 dP+\int_{J_{(2, \infty)}}(x-a(2, \infty))^2 dP=\frac{39}{14308}=0.00272575,\]
i.e., $V_8\geq 0.00272575>V_8$, which yields a contradiction. Suppose that $\te{card} (\ga_8\ii J_{(1, \infty)})=3$, and then taking $\gb_3=\set{a(2), a(3), a(3, \infty)}$, we see that
\[V_8\geq \int_{J_2}(x-a(2))^2 dP+\int_{J_3}(x-a(3))^2 dP+\int_{J_{(3, \infty)}}(x-a(3, \infty))^2 dP=\frac{93}{114464}=0.000812483>V_8,\]
which gives a contradiction. Next, assume that $\te{card} (\ga_8\ii J_{(1, \infty)})=4$, and then taking \[\gb_4=\set{a(21), a(21, \infty), a(3), a(3, \infty)},\] we see that
\[V_8\geq (1+\frac{43}{3})p_{21}s_{21}^2 V+(1+\frac{43}{9})p_3s_3^3V=\frac{831}{1831424}=0.000453745>V_8,\]
which gives a contradiction. So, we can assume that  $\te{card} (\ga_8\ii J_{(1, \infty)})\geq 5$. If $\te{card} (\ga_8\ii J_1)=1$, then,
\[V_8\geq \int_{J_1}(x-a(1))^2 dP=\frac{9}{7154}=0.00125804>V_8,\]
which leads to a contradiction. If $\te{card} (\ga_8\ii J_1)=2$, then taking $\gb_2=\set{a(11), a(11, \infty)}$, we see that
\[V_8\geq \int_{J_1} \min_{a\in \gb_2}(x-a)^2 dP=(1+\frac{43}{3})p_{11}s_{11}^2 V=\frac{69}{228928}=0.000301405>V_8,\]
which is a contradiction.
So, we can assume that $\te{card} (\ga_8\ii J_1)\geq 3$. Since $\te{card} (\ga_8\ii J_1)\geq 3$ and  $\te{card} (\ga_8\ii J_{(1, \infty)})\geq 5$, we have  $\te{card} (\ga_8\ii J_1)=3$ and $\te{card} (\ga_8\ii J_{(1, \infty)})=5$, which is the lemma.
\end{proof}

\begin{prop} \label{prop2} Let $\ga_n$ be an optimal set of $n$-means for $P$ such that $\te{card}(\ga_n\ii J_{(k, \infty)})\geq 2$ for some $k\in \D N$ and $n\in \D N$. Then, $\ga_n\ii J_{k+1}\neq \es$, $\ga_n\ii J_{(k+1, \infty)}\neq \es$, and $\ga_n$ does not contain any point from the open interval $(S_{k+1}(1),  S_{k+2}(0))$. Moreover, the Voronoi region of any point in $\ga_n\ii J_{k+1}$ does not contain any point from $J_{(k+1, \infty)}$ and the Voronoi region of any point in $\ga_n\ii J_{(k+1, \infty)}$ does not contain any point from $J_{k+1}$.
\end{prop}
\begin{proof} By Proposition~\ref{prop1}, since $\ga_n$ does not contain any point from $(\frac 14, \frac 12)$, the Voronoi region of any point in $\ga_n\ii J_1$ does not contain any point from $J_{(1, \infty)}$, and the Voronoi region of any point in $\ga_n\ii J_{(1, \infty)}$ does not contain any point from $J_1$, to prove the proposition it is enough to prove it for $k=1$, and then inductively the proposition will follow for all $k\geq 2$. Fix $k=1$. By Lemma~\ref{lemma12}, it is clear that the proposition is true for $n=3$. Let $\ga_4:=\set{0<a_1<a_2<a_3<a_4<1}$ be an optimal set of four-means. In the proof of Lemma~\ref{lemma13}, we have seen that $\frac 12\leq a_2$ yielding $\ga_4\ii J_{(1, \infty)}=\set{\frac 12\leq a_2<a_3<a_4<1}$, i.e., $\te{card}(\ga_4\ii J_{(1, \infty)})=3\geq 2$. We now prove the proposition for $n=4$. Let $V(P, \ga_4\ii J_{(1, \infty)})$ be the quantization error contributed by the set $\ga_4\ii J_{(1, \infty)}$. The distortion error due to the set $\gb:=\set{a(2), a(3), a(3, \infty)}$ of three points on $J_{(1, \infty)}$ is given by
\[\int_{J_{(1, \infty)}} \min_{a\in \gb}(x-a)^2 dP=p_2s_2^2V+(1+\frac {43}{9})p_3s_3^2V=\frac{93}{114464}=0.000812483,\]
and so $V(P, \ga_4\ii J_{(1, \infty)})\leq 0.000812483$. If $a_2\geq \frac {39}{64}=S_{24}(0)$, then
\[V(P, \ga_4\ii J_{(1, \infty)})\geq \int_{J_{21}\uu J_{22}\uu J_{23}}(x-\frac {39}{64})^2 dP=\frac{269769}{267911168}=0.00100693>V(P, \ga_4\ii J_{(1, \infty)}),\]
which is a contradiction. So, we can assume that $a_2< \frac {39}{64}$. Suppose that $a_3\leq \frac 57$. Then, as $S_{3}(1)=\frac{13}{16}<\frac 12(\frac 57+a(3, \infty))<\frac 78$, we have
\[V(P, \ga_4\ii J_{(1, \infty)})\geq \int_{J_{3}}(x-\frac 57)^2 dP+\int_{J_{(3, \infty)}}(x-a(3, \infty))^2 dP =\frac{297}{228928}=0.00129735\]
implying
$V(P, \ga_4\ii J_{(1, \infty)})\geq 0.00129735>V(P, \ga_4\ii J_{(1, \infty)})$, which is a contradiction. Next, suppose that $\frac 57\leq a_3\leq \frac 34$. Then, as $S_2(1)<\frac 12(a(2)+\frac 57)$ and $S_3(1)<\frac 12(\frac 34+a(3, \infty))<\frac 78=S_4(0)$, we have
\[V(P, \ga_4\ii J_{(1, \infty)})\geq \int_{J_{2}}(x-a(2))^2 dP+\int_{J_3}(x-\frac 34)^2 dP+\int_{J_{(3, \infty)}}(x-a(3, \infty))^2 dP =\frac{963}{915712}\]
yielding $V(P, \ga_4\ii J_{(1, \infty)})\geq\frac{963}{915712}= 0.00105164>V(P, \ga_4\ii J_{(1, \infty)}),$
which gives a contradiction. Thus, we have $\frac 34<a_3$. Since $a_2\leq \frac {39}{64}<\frac 58$ and $\frac 34<a_3$, the set $\ga_4\ii J_{(1, \infty)}$ does not contain any point from the open interval $(S_{2}(1), S_3(0))$.  Since $\frac 12(a_2+a_3)\geq \frac 12(\frac 12+\frac 34)=\frac 58=S_{2}(1)$, the Voronoi region of any point in $\ga_4\ii J_{(2, \infty)}$ does not contain any point from $J_2$. Suppose that the Voronoi region of any point in $\ga_4\ii J_2$ contains points from $J_{(2, \infty)}$. Then, $\frac 12(a_2+a_3)>\frac 34$ implying $a_3>\frac 32-a_2\geq \frac 32-\frac{39}{64}=\frac{57}{64}$, and so
\[V_4\geq \int_{J_3}(x-\frac{57}{64})^2 dP=\frac{10155}{4784128}=0.00212264>V_4,\]
which leads to a contradiction. Hence, the Voronoi region of any point in $\ga_4\ii J_2$ does not contain any point from $J_{(2, \infty)}$. Thus, the proposition is true for $n=4$. From the proof of Lemma~\ref{lemma14}, we see that if $\ga_5=\set{0<a_1<a_2<a_3<a_4<a_5<1}$ is an optimal set of five-means, then $\ga_5\ii J_{(1, \infty)}=\set{\frac 12\leq a_3<a_4<a_5<1}$. Thus, the proof of the proposition for $n=5$ follows exactly in the similar ways as the proof for $n=4$ given above.

Now, we prove the proposition for $n=6$. Let $\ga_6:=\set{0<a_1<a_2<a_3<a_4<a_5<a_6<1}$ be an optimal set of six-means. Then, by Lemma~\ref{lemma15}, we know that $\te{card} (\ga_6\ii J_2)=2$, and $\te{card} (\ga_6\ii J_{(1, \infty)})=4$. Thus, we see that $\ga_6\ii J_2=\set{a_3, a_4}\neq \es$ and $\ga_6\ii J_{(2, \infty)}=\set{a_5, a_6}\neq \es$. As shown in the proof of Lemma~\ref{lemma15},  we have $\ga_6\ii J_{(1, \infty)}=\set{\frac 12\leq a_3<a_4< a_5< a_6<1}$, and if $V(P, \ga_6\ii J_{(1, \infty)})$ is the quantization error contributed by the set $\ga_6\ii J_{(1, \infty)}$ in the region $J_{(1, \infty)}$, then we have
 $V(P, \ga_6\ii J_{(1, \infty)})\leq 0.000453745$. We now show that the Voronoi region of any point in $\ga_6\ii J_2$ does not contain any point from $J_{(2, \infty)}$. If it does, then we must have $\frac 12(a_4+a_5)>\frac 34$ implying $a_5>\frac 32-a_4\geq \frac 32-\frac 58=\frac 78$, and so
 \[V(P, \ga_6\ii J_{(1, \infty)})\geq \int_{J_3}(x-\frac 78)^2 dP=\frac{813}{523264}=0.00155371>V(P, \ga_6\ii J_{(1, \infty)}),\]
 which is a contradiction. Also, notice that the Voronoi region of any element from $\ga_6\ii J_{(2, \infty)}$ does not contain any point from $J_2$, if it does we must have $\frac 12(a_4+a_5)<\frac 58$ implying $a_4<\frac 54-a_5\leq \frac 54-\frac34=\frac 12$, which is a contradiction as $\frac 12\leq a_3<a_4$.

 Now, we prove the proposition for $n=7$.
Let $\ga_7:=\set{0<a_1<a_2<\cdots<a_7<1}$ be an optimal set of seven-means. By Lemma~\ref{lemma153}, first assume that $\te{card}(\ga_7\ii J_{(1,\infty)})=4$, i.e., $\frac 12\leq a_4$. Let $V(P, \ga_7\ii J_{(1, \infty)})$ be the quantization error contributed by the set $\ga_7\ii J_{(1, \infty)}$ in the region $J_{(1, \infty)}$. Let $\gb:=\set{a(11), a(12), a(12, \infty), a(21), a(21, \infty), a(3), a(3, \infty)}$. The distortion error due to the set $\gb\ii J_{(1, \infty)}:=\set{a(21), a(21, \infty), a(3), a(3, \infty)}$ is given by
\[\int_{J_{(1, \infty)}}\min_{a \in \gb\ii J_{(1, \infty)}}(x-a)^2dP=(1+\frac {43}{3})p_{21}s_{21}^2 V+(1+\frac {43}{9})p_3s_3^2 V=\frac{831}{1831424}=0.000453745,\]
and so $V(P, \ga_7\ii J_{(1, \infty)})\leq 0.000453745$. If $a_4\geq \frac{77}{128}=S_{23}(1)$, then
\[V(P, \ga_7\ii J_{(1, \infty)})\geq \int_{J_{21}\uu J_{22}\uu J_{23}}(x-\frac{77}{128})^2 dP=\frac{852849}{1071644672}=0.000795832>V(P, \ga_7\ii J_{(1, \infty)}),\]
which is a contradiction. So, we can assume that $a_4< \frac{77}{128}=S_{23}(1)$. Suppose that $\frac {11}{16}\leq a_5$. Then, as $\frac 12(a(2)+a_5)\geq \frac{1}{2} (\frac{4}{7}+\frac{11}{16})>\frac 58$, we have
\[V(P, \ga_7\ii J_{(1, \infty)})\geq \int_{J_2}(x-a(2))^2 dP=\frac{27}{57232}=0.000471764>V(P, \ga_7\ii J_{(1, \infty)}),\]
which leads to a contradiction. So, we can assume that $a_5\leq \frac{11}{16}$. Suppose that $\frac 58<a_5\leq \frac {11}{16}$. Then,  $\frac 12(a_5+a_6)>\frac 34$ implying $a_6>\frac 32-a_6\geq \frac 32-\frac {11}{16}=\frac{13}{16}=S_3(1)$. Then, the following two cases can arise:

 Case~(i). $\frac {27}{32}\leq a_6$.

 Then, $S_{31}(1)=\frac{49}{64}=\frac 12(\frac {11}{16}+\frac {27}{32})<S_{32}(0)$, and so by Lemma~\ref{lemma141},
\begin{align*}& V(P, \ga_7\ii J_{(1, \infty)})\\
 &\geq \int_{J_{21}\uu J_{22}} (x-a(21, 22))^2 dP+ \int_{J_{(22, \infty)}}(x-\frac 58)^2 dP+\int_{J_{31}}(x-\frac {11}{16})^2dP +\int_{J_{32}}(x-\frac {27}{32})^2 dP\\
 &=\frac{236721}{334888960}=0.000706864>V(P, \ga_7\ii J_{(1, \infty)}),
\end{align*}
 which gives a contradiction.

Case~(ii). $\frac {13}{16}<a_6<\frac {27}{32}$.

 Then, $\frac 12(a_6+a_7)>\frac 78$ implying $a_7>\frac 74-a_6\geq \frac 74-\frac {27}{32}=\frac {29}{32}=S_4(1)$. First, assume that $S_4(1)<a_7<S_5(0)=\frac{15}{16}$. Then, using Lemma~\ref{lemma141},
 \begin{align*}& V(P, \ga_7\ii J_{(1, \infty)})\\
 &\geq \int_{J_{21}\uu J_{22}} (x-a(21, 22))^2 dP+ \int_{J_{(22, \infty)}}(x-\frac 58)^2 dP+\int_{J_{3}}(x-\frac {13}{16})^2 dP+\int_{J_4}(x-\frac {29}{32})^2 dP\\
 &+\int_{J_5\uu J_6}(x-\frac{15}{16})^2dP=\frac{11529}{23920640}=0.000481969>V(P, \ga_7\ii J_{(1, \infty)}),
\end{align*}
which leads to a contradiction. Next, assume that $S_5(0)=\frac {15}{16}\leq a_7$. Then, as $S_{42}(0)=\frac{57}{64}=\frac 12(\frac{27}{32}+\frac {15}{16})$, using Lemma~\ref{lemma141},  we have
  \begin{align*}& V(P, \ga_7\ii J_{(1, \infty)})\\
 &\geq \int_{J_{21}\uu J_{22}} (x-a(21, 22))^2 dP+ \int_{J_{(22, \infty)}}(x-\frac 58)^2 dP+\int_{J_{3}}(x-\frac {13}{16})^2dP +\int_{J_{41}}(x-\frac {27}{32})^2 dP\\
 &+\int_{J_{42}}(x-\frac{15}{16})^2dP=\frac{700899}{1339555840}=0.000523232>V(P, \ga_7\ii J_{(1, \infty)}),
\end{align*}
 which yields a contradiction.

 Hence, by Case~(i) and Case~(ii), we can assume that $a_5\leq \frac 58$. If $a_6\leq \frac 34$, then as $\frac{13}{16}=S_{3}(1)=\frac 12(\frac 34+\frac 78)<\frac 12(\frac 34+a(3, \infty))=\frac 12(\frac 34+\frac {13}{14})<\frac 78$, we have
 \[V_7\geq \int_{J_3}(x-\frac 34)^2 dP+\int_{J_{(3, \infty)}}(x-a(3, \infty))dP=\frac{531}{915712}=0.000579877>V_7,\]
 which leads to a contradiction. So, we can assume that $\frac 34<a_6$. Thus, it is proved that $\ga_7\ii J_2\neq \es$, $\ga_7\ii J_{(2, \infty)}\neq \es$, and $\ga_7$ does not contain any point from the open interval $(S_2(1), S_3(0))$. Since $\frac 12(a_5+a_6)\geq \frac 12(\frac 12+\frac 34)=\frac 58$, the Voronoi region of any point in $\ga_7\ii J_{(2, \infty)}$ does not contain any point from $J_2$. If the Voronoi region of any point in $\ga_7\ii J_2$ contains points from $J_{(2, \infty)}$, we must have $\frac 12(a_5+a_6)>\frac 34$ implying $a_6>\frac 32-a_5\geq \frac 32-\frac 58=\frac 78$, and so
 \[V(P, \ga_7\ii J_{(1, \infty)})\geq \int_{J_3}(x-\frac{7}{8})^2dP=\frac{813}{523264}=0.00155371>V(P, \ga_7\ii J_{(1, \infty)}),\]
which is a contradiction. Thus, the Voronoi region of any point in $\ga_7\ii J_2$ does not contain any point from $J_{(2, \infty)}$ as well.

If we assume $\te{card}(\ga_7\ii J_{(1, \infty)})=5$, with the help of Lemma~\ref{lemma153}, similarly we can prove that the proposition is true. Notice that if we take $n=8$, then by Lemma~\ref{lemma1531}, we have $\te{card}(\ga_8\ii J_{(1, \infty)})=5$. Thus, the proof of the proposition for the case $n=8$ is exactly same as the proof of the proposition for $n=7$ with $\te{card}(\ga_7\ii J_{(1, \infty)})=5$.

Now, we prove the proposition for any $n\geq 9$.
Let $\ga_n:=\set{0<a_1<a_2<\cdots<a_n<1}$ be an optimal set of $n$-means for any $n\geq 9$ such that $\te{card}(\ga_n\ii J_{(1,\infty)})\geq 2$. Let $V(P, \ga_n\ii J_{(1, \infty)})$ be the quantization error contributed by the set $\ga_n\ii J_{(1, \infty)}$ in the region $J_{(1, \infty)}$. Let $\gb:=\set{a(11), a(12), a(12, \infty), a(21), a(22), a(22, \infty), a(3), a(4), a(4, \infty)}$. The distortion error due to the set $\gb\ii J_{(1, \infty)}:=\set{a(21), a(22), a(22, \infty), a(3), a(4), a(4, \infty)}$ is given by
\[\int_{J_{(1, \infty)}}\min_{a \in \gb\ii J_{(1, \infty)}}(x-a)^2dP=p_{21}s_{21}^2V+ (1+\frac {43}{9})p_{22}s_{22}^2 V+p_3s_3^2V +(1+\frac {43}{9})p_4s_4^2 V=\frac{915}{7325696},\]
and so $V(P, \ga_n\ii J_{(1, \infty)})\leq \frac{915}{7325696}=0.000124903$. Suppose that $\ga_n$ does not contain any point from $J_2$. Since by Proposition~\ref{prop1}, the Voronoi region of any point in $\ga_n\ii J_1$ does not contain any point from $J_{(1, \infty)}$, we have
 \[V(P, \ga_n\ii J_{(1, \infty)})\geq \int_{J_2}(x-\frac 58)^2 dP=\frac{405}{261632}=0.00154798>V(P, \ga_n\ii J_{(1, \infty)}),\]
 which leads to a contradiction. So, we can assume that $\ga_n\ii J_2\neq \es$.
Let $j:=\max\set{i : a_i \leq \frac 58 \te{ for all } 1\leq i\leq n}$, and so $a_j\leq \frac 58$.
We now show that $a_{j+1}\geq \frac 34$. Suppose that $\frac 58<a_{j+1}<\frac 34$.
 Then, the following two cases can arise:

Case~1. $\frac 58<a_{j+1}\leq \frac {11}{16}.$

Then, $\frac 12(a_{j+1}+a_{j+2})>\frac 34$ implying $a_{j+2}>\frac 32-a_{j+1}\geq \frac 32-\frac {11}{16}=\frac {13}{16}$, and so
\[V(P, \ga_n\ii J_{(1, \infty)})\geq \int_{J_3}(x-\frac {13}{16})^2 dP=\frac{405}{2093056}=0.000193497>V(P, \ga_n\ii J_{(1, \infty)}),\]
which is contradiction.

Case~2. $\frac {11}{16} \leq a_{j+1}< \frac 34.$

Then, $\frac 12(a_{j}+a_{j+1})<\frac 58$ implying $a_{j}<\frac 54-a_{j+1}\leq \frac 54-\frac {11}{16}=\frac 9{16}=S_{22}(0)$, and so
\[V(P, \ga_n\ii J_{(1, \infty)})\geq \int_{J_{22}\uu J_{23}\uu J_{24}}(x-\frac 9{16})^2 dP=\frac{99}{524288}=0.000188828>V(P, \ga_n\ii J_{(1, \infty)}),\]
which gives a contradiction.

Thus, we have proved that $\ga_n\ii J_{2}\neq \es$, $\ga_n\ii J_{(2, \infty)}\neq \es$, and $\ga_n$ does not contain any point from the open interval $(S_{2}(1),  S_{3}(0))$. Since $\frac 12(a_j+a_{j+1})\geq \frac 12(\frac 12+\frac 34)=\frac {5}{8}$, the Voronoi region of any point in $\ga_n\ii J_{(2, \infty)}$ does not contain any point from $J_2$. If the Voronoi region of any point in $\ga_n\ii J_2$ contains points from $J_{(2, \infty)}$, we must have $\frac 12(a_j+a_{j+1})>\frac 34$ implying $a_{j+1}>\frac 32-a_j\geq \frac 32-\frac 58=\frac{7}{8}$, and so
\[V(P, \ga_n\ii J_{(1, \infty)})\geq \int_{J_3}(x-\frac{7}{8})^2dP=\frac{813}{523264}=0.00155371>V(P, \ga_n\ii J_{(1, \infty)}),\]
which is a contradiction. Hence, the Voronoi region of any point in $\ga_n\ii J_2$ does not contain any point from $J_{(2, \infty)}$. Thus, the proof of the proposition is complete.
\end{proof}

\begin{prop}\label{prop3}
Let $\ga_n$ be an optimal set of $n$-means for $n\geq 2$. Then, there exists a positive integer $k$ such that  $\ga_n\ii J_j \neq \es$ for all $1\leq j\leq k$, and $\te{card}(\ga_n\ii J_{(k, \infty)})=1$. Moreover, if $n_j:=\te{card}(\ga_j)$, where $\ga_j:=\ga_n\ii J_j$, then $n=\sum_{j=1}^k n_j+1$, with
\[V_n=\left\{\begin{array} {ll}
p_1s_1^2 V+\frac{43}{3} p_1s_1^2 V \te{ if } k=1, & \\
\mathop\sum\limits_{j=1}^k p_js_j^2 V_{n_j}+\frac{43}{9} p_k s_k^2 V \te{ if } k\geq 2. &
\end{array} \right.\]
\end{prop}

\begin{proof} Proposition~\ref{prop1} says that if $\ga_n$ is an optimal set of $n$-means for $n\geq 2$, then $\ga_n\ii J_1\neq \es$, $\ga_n\ii J_{(1, \infty)}\neq \es$, and $\ga_n$ does not contain any point from the open interval $(S_1(1), S_2(0))$. Proposition~\ref{prop2} says that if $\te{card}(\ga_n\ii J_{(k, \infty)})\geq 2$ for some $k\in \D N$, then $\ga_n\ii J_{k+1}\neq \es$ and $\ga_n\ii J_{(k+1, \infty)}\neq \es$. Moreover, $\ga_n$ does not take any point from the open interval $(S_{k+1}(1), S_{k+2}(0))$. Thus, by Induction Principle, we can say that if $\ga_n$ is an optimal set of $n$-means for $n\geq 2$, then there exists a positive integer $k$ such that $\ga_n\ii J_j\neq \es$ for all $1\leq j\leq k$ and $\te{card}(\ga_n\ii J_{(k, \infty)})=1$.

For a given $n\geq 2$, write $\ga_j:=\ga_n\ii J_j$ and $n_j:=\te{card}(\ga_j)$. Since $\ga_j$ are disjoints for $1\leq j\leq k$, and $\ga_n$ does not contain any point from the open intervals $(S_{\ell}(1), S_{\ell+1}(0))$ for $1\leq \ell\leq k$, we have $\ga_n=\mathop{\uu}\limits_{j=1}^k\ga_j\uu \set{a(k, \infty)}$ and $n=n_1+n_2+\cdots +n_k+1$. Then, using Lemma~\ref{lemma1}, we deduce
\begin{align*} &V_n=\int \min_{a \in \ga_n}  \|x-a\|^2 dP=\sum_{j=1}^{k} \int_{J_j}\min_{a \in \ga_j} (x-a)^2 dP+\int_{J_{(k, \infty)}}(x-a(k, \infty))^2 dP\\
&=\sum_{j=1}^{k} p_j\int\min_{a \in \ga_j} (x-a)^2 d(P\circ S_j^{-1})+\int_{J_{(k, \infty)}}(x-a(k, \infty))^2 dP,
\end{align*}
which yields
\begin{equation} \label{eq33} V_n=\sum_{j=1}^{k}p_js_j^2\int\min_{a \in S_j^{-1}(\ga_j)} (x-a)^2 dP+\frac{43}{9} p_k s_k^2 V.\end{equation}
We now show that $S_j^{-1}(\ga_j)$ is an optimal set of $n_j$-means, where $1\leq j\leq k$. If $S_j^{-1}(\ga_j)$ is not an optimal set of $n_j$-means, then we can find a set $\gb \sci \D R$ with $\te{card}(\gb)=n_j$ such that $\int\mathop{\min}\limits_{b\in \gb} (x-b)^2 dP<\int \mathop{\min}\limits_{a\in S_j^{-1}(\ga_j)}(x-a)^2 dP$. But, then $S_j(\gb) \uu (\ga_n\setminus \ga_j)$ is a set of cardinality $n$ such that
\[\int\mathop{\min}\limits_{a\in S_j(\gb) \uu (\ga_n\setminus \ga_j)}(x-a)^2 dP<\int\mathop{\min}\limits_{a\in \ga_n} (x-a)^2 dP,\] which contradicts the optimality of $\ga_n$. Thus, $S_j^{-1}(\ga_j)$ is an optimal set of $n_j$-means for $1\leq j\leq k$. Hence, by \eqref{eq33} we have
\[V_n=\sum_{j=1}^{k}p_js_j^2V_{n_j}+\frac{43}{9} p_k s_k^2 V.\]
Thus, the proof of the proposition is yielded.
\end{proof}

We need the following lemma to prove the main theorem Theorem~\ref{Th1} of the paper.

\begin{lemma} \label{lemma16}
For any $\go \in \D N^k$, $k\geq 1$, let $E(a(\go))$ and $E(a(\go, \infty))$ be given by \eqref{eq43}. Then, for $\go, \gt \in \D N^k$, $k\geq 1$, we have

$(i)$ $E(a(\go))> E(a(\gt))$ if and only if $E(a(\go1))+E(a(\go 1, \infty))+E(a(\gt))< E(a(\go))+E(a(\gt1))+E(a(\gt 1, \infty))$;

$(ii)$ $E(a(\go))> E(a(\gt, \infty))$ if and only if $E(a(\go1))+E(a(\go 1, \infty))+E(a(\gt, \infty))< E(a(\go))+E(a(\gt^-(\gt_{|\gt|}+1)))+E(a(\gt^-(\gt_{|\gt|}+1), \infty))$;

$(iii)$  $E(a(\go, \infty))> E(a(\gt))$ if and only if
$E(a(\go^-(\go_{|\go|}+1)))+E(a(\go^-(\go_{|\go|}+1), \infty))+E(a(\gt))< E(a(\go, \infty))+E(a(\gt1))+E(a(\gt 1, \infty))$;

$(iv)$  $E(a(\go, \infty))> E(a(\gt, \infty))$ if and only if
$E(a(\go^-(\go_{|\go|}+1)))+E(a(\go^-(\go_{|\go|}+1), \infty))+E(a(\gt, \infty))< E(a(\go, \infty))+E(a(\gt^-(\gt_{|\gt|}+1)))+E(a(\gt^-(\gt_{|\gt|}+1), \infty))$.
\end{lemma}

\begin{proof} To prove $(i)$, using Lemma~\ref{lemma4}, we see that

\begin{align*}
LHS&=E(a(\go1))+E(a(\go 1, \infty))+E(a(\gt))=p_{\go1}s_{\go1}^2V(1+\frac {43}{3})+p_\gt s_\gt^2 V\\
&=\frac 1{64}p_\go s_\go^2 V(1+\frac{43}{3})+p_\gt s_\gt^2 V,\\
RHS&=E(a(\go))+E(a(\gt1))+E(a(\gt 1, \infty))=p_\go s_\go^2 V+\frac 1{64}p_{\gt} s_{\gt}^2V (1+\frac {43}{3}).
\end{align*}
Thus, $LHS< RHS$ if and only if $\frac 1{64}p_\go s_\go^2 V(1+\frac{43}{3})+p_\gt s_\gt^2 V< p_\go s_\go^2 V+\frac 1{64}p_{\gt} s_{\gt}^2V (1+\frac {43}{3})$, which yields $p_\go s_\go^2 V > p_\gt s_\gt^2 V$, i.e., $E(a(\go))> E(a(\gt))$. Thus $(i)$ is proved.
To prove $(ii)$, let us first assume $\gt_{|\gt|}=1$. Notice that  $p_{\gt^-(\gt_{|\gt|}+1)}=p_{\gt^-} p_{\gt_{|\gt|}+1}=\frac 3 2 p_\gt$, and  $s_{\gt^-(\gt_{|\gt|}+1)}=s_{\gt^-} s_{\gt_{|\gt|}+1}=\frac 1 2 s_\gt$, and then using Lemma~\ref{lemma4}, we have
\begin{align*}
LHS&=E(a(\go1))+E(a(\go 1, \infty))+E(a(\gt, \infty))=p_{\go1}s_{\go1}^2V(1+\frac {43}{3})+\frac{43}{3} p_\gt s_\gt^2 V\\
&=\frac 1{64}p_\go s_\go^2 V(1+\frac{43}{3})+\frac{43}{3} p_\gt s_\gt^2 V,\\
RHS&=E(a(\go))+E(a(\gt^-(\gt_{|\gt|}+1)))+E(a(\gt^-(\gt_{|\gt|}+1), \infty))\\
&=p_\go s_\go^2 V+p_{\gt^-(\gt_{|\gt|}+1)}s_{\gt^-(\gt_{|\gt|}+1)}^2 V(1+\frac{43}{9})=p_\go s_\go^2 V+p_{\gt}s_{\gt}^2 V \frac 38(1+\frac{43}{9}).
\end{align*}
 Thus, $LHS< RHS$ if and only if $\frac 1{64}p_\go s_\go^2 V(1+\frac{43}{3})+\frac{43}{3} p_\gt s_\gt^2 V< p_\go s_\go^2 V+p_{\gt}s_{\gt}^2 V \frac 38(1+\frac{43}{9})$, which yields
 \[p_\go s_\go^2 V > \frac {43}{3} p_\gt s_\gt^2 V \frac{ \Big(\frac{43}{3}-\frac 38 (1+\frac{43}{9})\Big) \frac {3}{43}}{1-\frac 1 {64}(1+\frac {43}{3})}> \frac {43}{3} p_\gt s_\gt^2 V,\] i.e., $E(a(\go))> E(a(\gt, \infty))$. Thus, $(ii)$ is proved under the assumption $\gt_{|\gt|}=1$. Similarly by taking $\gt_{|\gt|}\geq  2$, we can prove $(ii)$. Thus, the proof of $(ii)$ is complete. Proceeding in the similar way, $(iii)$ and $(iv)$ can be proved. This concludes the proof of the lemma.
\end{proof}

The following proposition gives some properties of $E(\go)$ for $\go\in \D N^\ast$.

\begin{prop} \label{prop10}
Let $\go, \gt$ be two nonempty words in $\D N^\ast$ with $p_\go=p_\gt$. Then, the quantization error satisfies the following conditions:

$(i)$ $E(a(\go))=E(a(\gt))$.

$(ii)$ If $\go_{|\go|}=\gt_{|\gt|}$, then $E(a(\go, \infty))=E(a(\gt, \infty))$.

$(iii)$ If $\go_{|\go|}\neq \gt_{|\gt|}=1$, then $E(a(\go, \infty))=\frac 13E(a(\gt, \infty))$.

$(iv)$ If $1=\go_{|\go|}\neq \gt_{|\gt|}$, then $E(a(\go, \infty))=3 E(a(\gt, \infty))$.
\end{prop}
\begin{proof}

$(i)$ By Lemma~\ref{lemma5}, $p_\go=p_\gt$ implies $s_\go=s_\gt$, and so
\[E(a(\go))=p_\go s_\go^2V=p_\gt s_\gt^2 V=E(a(\gt)).\]

$(ii)$ Here two cases can arise: $\go_{|\go|}=\gt_{|\gt|}=1$ or $\go_{|\go|}=\gt_{|\gt|}\geq 2$. In either case, using Lemma~\ref{lemma4} one can see that
$E(a(\go, \infty))=E(a(\gt, \infty))$.

$(iii)$ If $\go_{|\go|}\neq \gt_{|\gt|}=1$, then, $\go_{|\go|}\geq 2$ and $\gt_{|\gt|}=1$, and so by Lemma~\ref{lemma4} and Lemma~\ref{lemma5}, we get
\[E(a(\go,\infty))=\frac{43}{9} p_\go s_\go^2 V=\frac 13 \frac {43}{3} p_\gt s_\gt^2 V=\frac 13 E(a(\gt, \infty)).  \]
Due to symmetry $(iv)$ follows from $(iii)$, and thus the proof of the proposition is complete.
\end{proof}

\begin{prop} \label{prop50}
Let $\ga_n$ be an optimal set of $n$-means for $n\geq 2$. Then, for $c\in \ga_n$, we have $c=a(\go)$, or $c=a(\go, \infty)$ for some $\go \in \D N^\ast$.
\end{prop}
\begin{proof} Let $\ga_n$ be an optimal set of $n$-means for $n\geq 2$ such that $c\in \ga_n$.
By Proposition~\ref{prop2}, there exists a positive integer $k_1$ such that  $\ga_n\ii J_{j_1}\neq \es$ for $1\leq j_1\leq k_1$,  and $\te{card}(\ga_n \ii J_{(k_1, \infty)})=1$, and $\ga_n$ does not contain any point from the open intervals $(S_{\ell}(1), S_{\ell+1}(0))$ for $1\leq \ell\leq k_1$. If $c\in \ga_n \ii J_{(k_1, \infty)}$, then $c=a(k_1, \infty)$. If $c\in \ga_n\ii J_{j_1}$ for some $1\leq j_1\leq k_1$ with $\te{card}(\ga_n\ii J_{j_1})=1$, then $c=a(j_1)$. Suppose that $c \in \ga_n\ii J_{j_1}$ for some $1\leq j_1\leq k_1$ and $\te{card}(\ga_n\ii J_{j_1})\geq 2$. Then, as similarity mappings preserve the ratio of the distances of a point from any other two points, using Proposition~\ref{prop2} again, there exists a positive integer $k_2$ such that  $\ga_n\ii J_{j_1j_2}\neq \es$ for $1\leq j_2\leq k_2$, and $\te{card}(\ga_n \ii J_{(j_1k_2, \infty)})=1$, and $\ga_n$ does not contain any point from the open intervals $(S_{j_1\ell}(1), S_{j_1(\ell+1)}(0))$ for $1\leq \ell\leq k_2$. If $c\in \ga_n \ii J_{(j_1k_2, \infty)}$ then $c=a(j_1k_2, \infty)$. Suppose that $c\in \ga_n\ii J_{j_1j_2}$ for some $1\leq j_2\leq k_2$. If $\te{card}(\ga_n\ii J_{j_1j_2})=1$, then $c=a(j_1j_2)$. If $\te{card}(\ga_n\ii J_{j_1j_2})\geq 2$, proceeding inductively as before, we can find a word $\go \in \D N^\ast$, such that either $c \in \ga_n\ii J_\go$ with $\te{card}(\ga_n\ii J_\go)=1$ implying $c=a(\go)$, or  $c\in \ga_n\ii J_{(\go, \infty)}$ with $\te{card}(\ga_n\ii J_{(\go, \infty)})=1$ implying $c=a(\go, \infty)$. Thus, the proof of the proposition is complete.
\end{proof}

By Proposition~\ref{prop50}, we can say that if $\ga_n$ is an optimal set of $n$-means for any $n\geq 2$, then the error contributed by any element $c\in \ga_n$ is given by $E(a(\go))$ if $c=a(\go)$, or by $E(a(\go, \infty))$ if $c=a(\go, \infty)$, where $\go \in \D N^\ast$. We are now ready to give the proof of Theorem~\ref{Th1}.

\subsection*{Proof of Theorem~\ref{Th1}}
By Lemma~\ref{lemma11} and Lemma~\ref{lemma12}, it is known that the optimal sets of two- and three-means are $\{a(1), a(1, \infty)\}$ and
$\{a(1), a(2), a(2, \infty)\}$.  Since \[E(a(1, \infty))=\frac{43}{3} p_1s_1^2V>p_1s_1^2 V=E(a(1)),\]
the theorem is true for $n=2$. For $n\geq 2$, let $\ga_n$ be an optimal
set of $n$-means. Let $\ga_n:=\set{a(i) : 1\leq i\leq n}$. Let $\tilde  E(a(i))$ and $W(\ga_n)$ be defined as in the hypothesis. If $a(j) \not \in W(\ga_n)$, i.e., if  $a(j) \in \ga_n\setminus W(\ga_n)$, then by Lemma~\ref{lemma16}, the error
\[\sum_{a(i)\in (\ga_n\setminus \set{a(j)})}E(a(i))+E(a(\go^-(\go_{|\go|}+1)))+E(a(\go^-(\go_{|\go|}+1), \infty)) \te{ if } a(j)=a(\go, \infty),\]
or
\[\sum_{a(i)\in (\ga_n\setminus \set{a(j)})}E(a(i))+E(a(\go1))+E(a(\go1, \infty)) \te{ if } a(j)=a(\go),\]
obtained in this case is strictly greater than the corresponding error obtained in the case when $a(j)\in W(\ga_n)$. Hence for any $a(j) \in W(\ga_n)$, the set $\ga_{n+1}(a(j))$, where
\[\ga_{n+1}(a(j))=\left\{\begin{array}{ll}
(\ga_n\setminus \set{a(j)})\uu \set{a(\go^-(\go_{|\go|}+1)), \, a(\go^-(\go_{|\go|}+1), \infty)} \te{ if } a(j)=a(\go, \infty), &\\
(\ga_n \setminus \set{a(j)})\uu \set{a(\go1), \, a(\go1, \infty)} \te{ if } a(j)=a(\go), &
\end{array}\right.
\] is an optimal set of $(n+1)$-means, and the number
of such sets is
\[\te{card}\Big(\UU_{\ga_n \in \C{C}_n}\{\ga_{n+1}(a(j)) : a(j) \in W(\ga_n)\}\Big).\]
Thus, the proof of the theorem is complete.
\qed

\begin{figure}
\centerline{\includegraphics[ ]{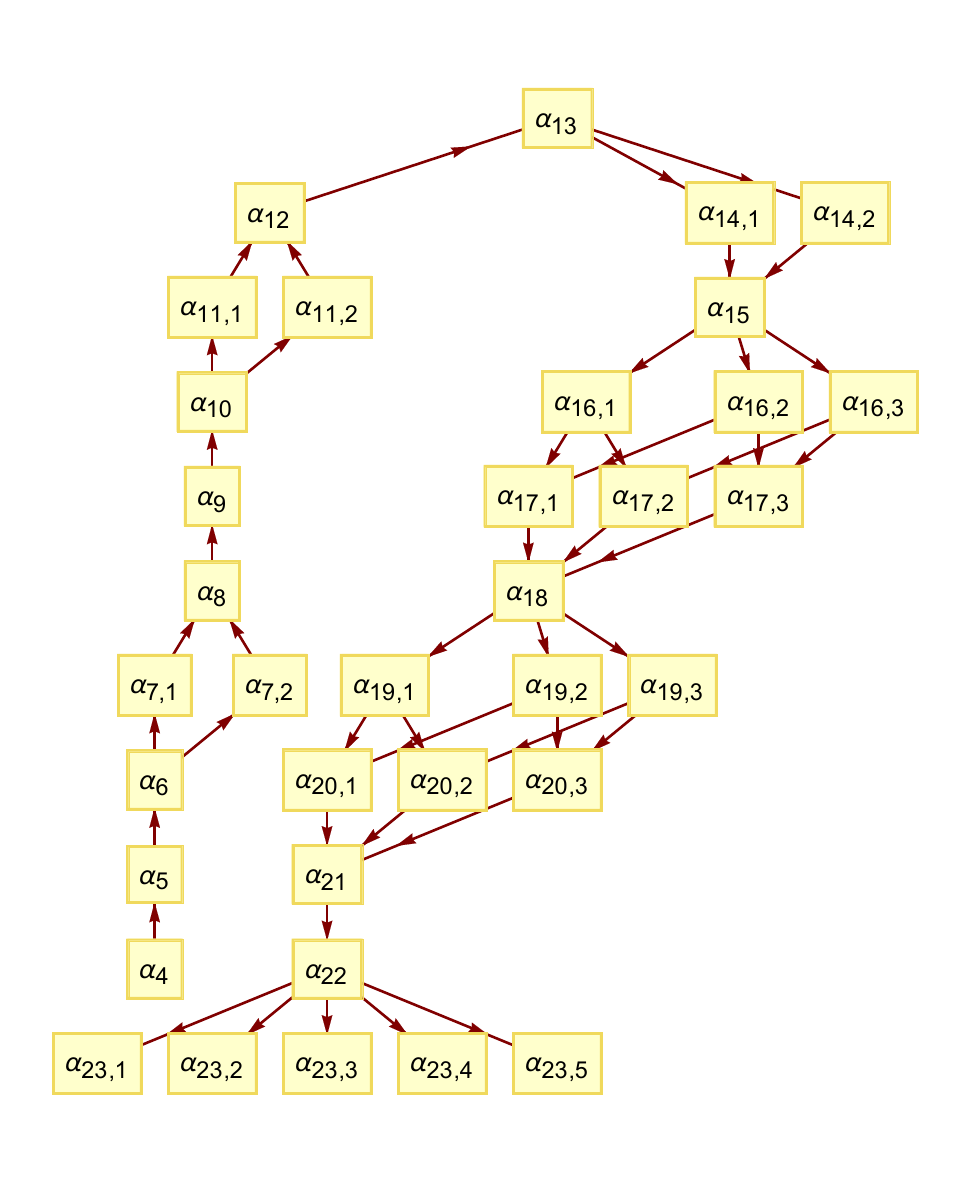}}
\caption{Tree diagram of the optimal sets from $\ga_4$ to $\ga_{23}$. \label{Fig2}}
\end{figure}

\medskip
The following results and observations are due to the induction formula given by Theorem~\ref{Th1}.

\section{Results and observations about optimal sets of $n$-means}

Recall that the optimal set of one-mean consists of the expected value of the random variable $X$, and the corresponding quantization error is its variance. Let $\ga_n$ be an optimal set of $n$-means, i.e., $\ga_n \in \C C_n$, and then for any $a\in \ga_n$, we have $a=a(\go)$ or $a=a(\go, \infty)$ for some $\go \in \D N^\ast$.
Theorem~\ref{Th1} implies that if $\te{card}(\C C_n)=k$ and  $\te{card}(\C C_{n+1})=m$, then either $1\leq k\leq m$, or $1\leq m\leq k$, for example from Figure~\ref{Fig2}, we see that the number of $\ga_{15}=1$, the number of $\ga_{16}=3$, the number of $\ga_{17}=3$, and the number of $\ga_{18}=1$. Thus, there exists a sequence $\set{n_k}_{k=1}^\infty$ of positive integers such that for all $n\geq 1$, we have $\te{card}(\C C_n)=n_k$, and then we write
\[\C C_n=\left\{\begin{array} {cc}
\set{\ga_n} & \te{ if } n_k=1, \\
\set{\ga_{n, i} : 1\leq i\leq n_k} & \te{ if } n_k\geq 2.
\end{array} \right.\]
In addition, Theorem~\ref{Th1} implies that a single $\ga\in \C C_n$ can produce multiple distinct $\ga\in \C C_{n+1}$, and multiple distinct $\ga\in \C C_n$ can produce one common $\ga \in \C C_{n+1}$. For $\ga \in \C C_n$, by $\ga\rightarrow \gb$, it is meant that $\gb \in \C C_{n+1}$ and $\gb$ is produced from $\ga$.
Thus, from Figure~\ref{Fig2}, we see that
 \begin{align*} &\left\{\ga _{18}\to \ga _{19,1},\ga _{18}\to \ga _{19,2}, \ga _{18}\to \ga _{19,3}\right\},\\
 &\left\{\left\{\ga_{19,1}\to \ga _{20,1}, \ga_{19,1}\to \ga_{20,2}\right\}, \left\{\ga _{19,2}\to \ga _{20,1},\ga _{19,2}\to \ga _{20,3}\right\},\left\{\ga _{19,3}\to \ga _{20,2},\ga _{19,3}\to \ga _{20,3}\right\}\right\},\\
  &\left\{\ga _{20,1}\to \ga _{21},\ga _{20,2}\to \ga _{21},\ga _{20,3}\to \ga _{21}\right\}.
 \end{align*}
Again, we have
\begin{align*}
\ga_{15}&=\set{a(111), a(111, \infty), a(12), a(13), a(13, \infty), a(21), a(22), a(23), a(23,\infty), a(31), a(32), \\
& \qquad \qquad a(32, \infty), a(4), a(5), a(5, \infty)}
\te{ with } V_{15}=\frac{27}{598016}=0.0000451493;\\
\ga_{16, 1}&=\set{a(111), a(111, \infty), a(12), a(13), a(13, \infty), a(211), a(211, \infty), a(22), a(23), a(23,\infty), \\
&  a(31), a(32), a(32, \infty), a(4), a(5), a(5, \infty)};\\
\ga_{16, 2}&=\set{a(111), a(111, \infty), a(12), a(13), a(13, \infty), a(21), a(22), a(23), a(23,\infty), a(31), a(32), \\
&  a(32, \infty), a(41), a(41, \infty), a(5), a(5, \infty)} \\
\ga_{16, 3}&=\set{a(111), a(111, \infty), a(121), a(121, \infty), a(13), a(13, \infty), a(21), a(22), a(23), a(23,\infty),  \\
&  a(31), a(32), a(32, \infty), a(4), a(5), a(5, \infty)}  \te{ with } V_{16}=\frac{4635}{117211136}=0.000039544; \\
\ga_{17, 1}&=\set{a(111), a(111, \infty), a(12), a(13), a(13, \infty), a(211), a(211, \infty), a(22), a(23), a(23,\infty), \\
& a(31), a(32), a(32, \infty), a(41), a(41, \infty), a(5), a(5, \infty)};\\
\ga_{17, 2}&=\set{a(111), a(111, \infty), a(121), a(121, \infty), a(13), a(13, \infty), a(211), a(211, \infty), a(22), a(23), \\
& a(23,\infty),   a(31), a(32), a(32, \infty), a(4), a(5), a(5, \infty)}.
\end{align*}
\begin{align*}
\ga_{17, 3}&=\set{a(111), a(111, \infty), a(121), a(121, \infty), a(13), a(13, \infty), a(21), a(22), a(23), a(23,\infty),  \\
& a(31), a(32), a(32, \infty), a(41), a(41, \infty), a(5), a(5, \infty)} \te{ with } V_{17}=\frac{1989}{58605568}=0.0000339388; \\
\ga_{18}&=\set{a(111), a(111, \infty), a(121), a(121, \infty), a(13), a(13, \infty), a(211), a(211, \infty), a(22), a(23), \\
&  \qquad \qquad a(23,\infty), a(31), a(32), a(32, \infty), a(41), a(41, \infty), a(5), a(5, \infty)} \\
&  \qquad \qquad \te{ with } V_{18}=\frac{3321}{117211136}=0.0000283335;
\end{align*}
and so on.

\end{document}